\documentclass[12pt]{amsart}
\usepackage{fullpage,url,mathrsfs,stmaryrd,amssymb}
\usepackage[all]{xy}

\newtheorem{theorem}{Theorem}[section]
\newtheorem{lemma}[theorem]{Lemma}

\newtheorem{cor}[theorem]{Corollary}
\newtheorem{conj}[theorem]{Conjecture}
\theoremstyle{definition}
\newtheorem{hypothesis}[theorem]{Hypothesis}
\newtheorem{notation}[theorem]{Notation}
\newtheorem{remark}[theorem]{Remark}
\newtheorem{defn}[theorem]{Definition}
\newtheorem{example}[theorem]{Example}

\numberwithin{equation}{theorem}

\newcommand{\AAA}{\mathbb{A}}
\newcommand{\CC}{\mathbb{C}}
\newcommand{\FF}{\mathbb{F}}
\newcommand{\PP}{\mathbb{P}}
\newcommand{\QQ}{\mathbb{Q}}
\newcommand{\RR}{\mathbb{R}}
\newcommand{\ZZ}{\mathbb{Z}}
\newcommand{\calE}{\mathcal{E}}
\newcommand{\calF}{\mathcal{F}}
\newcommand{\calG}{\mathcal{G}}
\newcommand{\calL}{\mathcal{L}}
\newcommand{\calO}{\mathcal{O}}
\newcommand{\calR}{\mathcal{R}}

\newcommand{\be}{\mathbf{e}}
\newcommand{\bv}{\mathbf{v}}
\newcommand{\dual}{\vee}

\DeclareMathOperator{\FIsoc}{\mathbf{F-Isoc}}

\DeclareMathOperator{\codim}{codim}
\DeclareMathOperator{\crys}{crys}

\DeclareMathOperator{\Ext}{Ext}
\DeclareMathOperator{\Frac}{Frac}
\DeclareMathOperator{\Frob}{Frob}

\DeclareMathOperator{\GL}{GL}
\DeclareMathOperator{\Hom}{Hom}

\DeclareMathOperator{\perf}{perf}

\DeclareMathOperator{\rank}{rank}
\DeclareMathOperator{\rig}{rig}
\DeclareMathOperator{\SL}{SL}
\DeclareMathOperator{\Spec}{Spec}
\DeclareMathOperator{\Swan}{Swan}
\DeclareMathOperator{\Sym}{Sym}
\DeclareMathOperator{\unr}{unr}

\begin{document}

\title{Notes on isocrystals}
\author{Kiran S. Kedlaya}
\date{November 20, 2021}
\thanks{These notes are based on lectures given in the geometric Langlands seminar at the University of Chicago during spring 2016. Thanks to Tomoyuki Abe, Marco D'Addezio, Valentina Di Proietto, Vladimir Drinfeld, H\'el\`ene Esnault, Ambrus P\'al, and Atsushi Shiho for additional feedback, and to D'Addezio for the reference \cite{yu}. The author was supported by NSF grants DMS-1501214, DMS-1802161, DMS-2053473, and the UC San Diego Warschawski Professorship.}

\begin{abstract}
For varieties over a perfect field of characteristic $p$, \'etale cohomology with $\QQ_\ell$-coefficients is a Weil cohomology theory only when $\ell \neq p$; the corresponding role for $\ell = p$ is played by Berthelot's rigid cohomology. In that theory, the coefficient objects analogous to lisse $\ell$-adic sheaves are the overconvergent $F$-isocrystals. This expository article is a brief user's guide for these objects, including some features shared with $\ell$-adic cohomology (purity, weights) and
some features exclusive to the $p$-adic case (Newton polygons, convergence and overconvergence). The relationship between the two cases, via the theory of companions, will be treated in sequel papers.
\end{abstract}

\maketitle

\section{Introduction}

Let $k$ be a perfect field of characteristic $p>0$. For each prime $\ell \neq p$, \'etale cohomology with $\QQ_\ell$-coefficients constitutes a Weil cohomology theory for varieties over $k$, in which the coefficient objects of locally constant rank are the \emph{smooth (lisse) $\overline{\QQ}_\ell$-local systems}; when $k$ is finite, one also considers \emph{lisse Weil $\overline{\QQ}_\ell$-sheaves}. This article is a brief user's guide for the $p$-adic analogues of these constructions; we focus on basic intuition and statements of theorems, omitting essentially all proofs (except for a couple of undocumented variants of existing proofs, which we record in an appendix).

To obtain a Weil cohomology with $p$-adic coefficients, Berthelot defined the theory of \emph{rigid cohomology}. One tricky aspect of rigid cohomology is that it includes not one, but two analogues of the category of smooth $\ell$-adic sheaves: the category of \emph{convergent $F$-isocrystals} and the subcategory of \emph{overconvergent $F$-isocrystals}. The former category can be interpreted in terms of crystalline sites (see Theorem~\ref{T:isogeny category}), but the latter can only be described using analytic geometry. (We will implicitly use rigid analytic geometry, but any of the other flavors of analytic geometry over nonarchimedean fields can be used instead.)

The distinction between convergent and overconvergent $F$-isocrystals carries important functional load: overconvergent $F$-isocrystals seem to be the objects which are ``classically motivic'' whereas convergent $F$-isocrystals can arise from geometric constructions exclusive to characteristic $p$. For example, the
``crystalline companion'' to a compatible system of lisse Weil $\overline{\QQ}_\ell$-sheaves (i.e., the ``petit camarade cristalline'' in the sense of \cite[Conjecture~1.2.10]{deligne-weil2}) is an overconvergent $F$-isocrystal, which is irreducible if the $\ell$-adic objects are; however, in the category of convergent $F$-isocrystals the crystalline companion often acquires a nontrivial \emph{slope filtration}. A typical example is provided by the cohomology of a universal family of elliptic curves
(Example~\ref{E:elliptic unit-root}).

When transporting arguments from $\ell$-adic to $p$-adic cohomology, one can often assign the role of $\QQ_\ell$-local systems appropriately to either convergent or overconvergent $F$-isocrystals. In a few cases, one runs into difficulties because neither category seems to provide the needed features; on the other hand, in some cases the rich interplay between the constructions makes it possible to transport statements back to the $\ell$-adic side which do not seem to have any direct proof there.

One can continue the story by describing links between $\ell$-adic and $p$-adic coefficients via the theory of \emph{companions} as alluded to above. However, this would require setting aside the premise of a purely expository paper, as some new results would be required. We have thus chosen to defer this discussion to two sequel papers \cite{kedlaya-companions, kedlaya-companions2}.

\begin{notation}
Throughout this paper, let $k$ denote a perfect field of characteristic $p>0$ (as above), and let $X$ denote a smooth variety over $k$. By convention, we require varieties to be reduced separated schemes of finite type over $k$, but they need not be irreducible.
Let $K$ denote the fraction field of the ring of $p$-typical Witt vectors $W(k)$.
\end{notation}

\section{The basic constructions}
\label{sec:basic construction}

We begin by illustration the construction of convergent and overconvergent $F$-isocrystals
on smooth varieties, following Berthelot's original approach to rigid cohomology in which the constructions are fairly explicit but not overtly functorial. A more functorial approach, using suitably constructed sites, is described in \cite{lestum}, to which we defer for justification of all unproved claims (and for treatment of nonsmooth varieties).

We will use without comment the fact that coherent sheaves on affinoid spaces correspond to finitely generated modules over the ring of global sections (i.e., Kiehl's theorem in rigid analytic geometry). See for example \cite[Chapter~9]{bgr}.

\begin{defn} \label{D:convergent realization}
For $X$ affine, we construct the category $\FIsoc(X)$ of \emph{convergent $F$-iso\-crystals on $X$} as follows. Using a lifting construction of Elkik \cite{elkik} (or its generalization by Arabia \cite{arabia}), we can find a smooth affine formal scheme $P$ over $W(k)$ with special fiber $X$ and a lift $\sigma: P \to P$ of the absolute Frobenius on $X$. Let $P_K$ denote the Raynaud generic fiber of $P$ as a rigid analytic space over $K$.
Then an object of $\FIsoc(X)$ is a vector bundle $\calE$ on $P_K$ equipped with an integrable connection (i.e., an $\calO$-coherent $\mathcal{D}$-module) and an isomorphism $\sigma^* \calE \cong \calE$ of $\mathcal{D}$-modules (which we view as a semilinear action of $\sigma$ on $\calE$); a morphism in $\FIsoc(X)$ is a $\sigma$-equivariant morphism of $\mathcal{D}$-modules.

One checks as in \cite{lestum}
(by comparing to a more functorial definition)
that the functor $\FIsoc$ is a stack for the Zariski and \'etale topologies on $X$.
This leads to a definition of $\FIsoc(X)$ for arbitrary $X$. When $X = \Spec R$ is affine, we will occasionally write $\FIsoc(R)$ instead of $\FIsoc(\Spec R)$.
\end{defn}

\begin{theorem}[Ogus] \label{T:isogeny category}
Let $\mathcal{C}$ be the isogeny category associated to the category of crystals of finite $\calO_{X,\mathrm{crys}}$-modules.
Then $\FIsoc(X)$ is canonically equivalent to the category of objects of $\mathcal{C}$ equipped with $F$-actions (i.e., isomorphisms with their $F$-pullbacks).
\end{theorem}
\begin{proof}
The functor from crystals to $\FIsoc(X)$ is exhibited in \cite{ogus} and shown therein to be fully faithful. For essential surjectivity, see \cite[Th\'eor\`eme~2.4.2]{berthelot-rigid}.
\end{proof}

\begin{remark}  \label{R:not locally free}
Theorem~\ref{T:isogeny category} implies that the category $\FIsoc(X)$ is abelian. This can also be seen more directly from the fact that (because $K$ is of characteristic zero) any coherent sheaf on a rigid analytic space over $K$ admitting a connection is automatically locally free. 
(See \cite[Proposition~1.2.6]{kedlaya-goodformal1} for a general argument to this effect.)

Even so, a general object of $\FIsoc(X)$ need not correspond to a crystal of \emph{locally free} $\calO_{X,\mathrm{crys}}$-modules, except in the unit-root case (see Theorem~\ref{T:unit root1} below). However, using the fact that reflexive modules on regular schemes are locally free in dimension 2, one sees that for $\calE \in \FIsoc(X)$, there exists an open dense subspace $U$ of $X$ with $\codim(X-U, X) \geq 2$ for which the restriction of $\calE$ to $\FIsoc(U)$ can be realized as a crystal of locally free $\calO_{X,\mathrm{crys}}$-modules.
(See \cite[Lemma~2.5.1]{crew-f} for a detailed discussion.)
In some cases, one can promote the desired results from $U$ back to $X$ using purity for isocrystals; see Theorem~\ref{T:purity}.
\end{remark}

\begin{defn} \label{D:overconvergent realization}
For $X \to Y$ an open immersion of $k$-varieties with $X$ and $Y$ affine (but $Y$ not necessarily smooth), we construct the category $\FIsoc(X,Y)$ of \emph{isocrystals on $X$ overconvergent within $Y$} as follows. Again using the results of Elkik or Arabia, we can find an affine formal scheme $P$ over $W(k)$ with special fiber $Y$ which is smooth in a neighborhood of $X$ and a lift $\sigma: Q \to Q$ of absolute Frobenius, for $Q$ the open formal subscheme of $P$ supported on $Y$, which extends to a neighborhood of $Q_K$ in $P_K$ for the Berkovich topology (or in more classical terminology, a \emph{strict neighborhood} of $Q_K$ in $P_K$). Then an object of $\FIsoc(X,Y)$ is a vector bundle $\calE$ on some strict neighborhood equipped with an integrable connection and an isomorphism $\sigma^* \calE \cong \calE$ of $\mathcal{D}$-modules; a morphism in $\FIsoc(X,Y)$ is a $\sigma$-equivariant morphism of $\mathcal{D}$-modules defined on some strict neighborhood of $Q_K$, with two morphisms considered equal if they agree on some (hence any) strict neighborhood on which they are both defined. In particular, restriction of a bundle from one strict neighborhood to another is an isomorphism in $\FIsoc(X,Y)$.

One again checks as in \cite{lestum}
that the functor $\FIsoc$ is a stack for the Zariski and \'etale topologies on $Y$.
This leads to a definition of $\FIsoc(X,Y)$ for an arbitrary open immersion $X \to Y$.
\end{defn}

\begin{remark} \label{R:pullback functoriality}
Given a commutative diagram
\[
\xymatrix{
X' \ar[r] \ar[d] & Y' \ar[d] \\
X \ar[r] & Y
}
\]
in which $X' \to Y'$ is again an open immersion of $k$-varieties with $X'$ smooth,
one obtains a pullback functor $\FIsoc(X, Y) \to \FIsoc(X', Y')$. 
If $X' = X$, then this pullback functor is obviously faithful;
we will see later that it is also full (Theorem~\ref{T:fully faithful1}).
\end{remark}

\begin{lemma}[Berthelot] \label{L:change compactification}
Let $f: Y' \to Y$ be a proper morphism such that $f^{-1}(X) \to X$ is an isomorphism.
Then the pullback functor $\FIsoc(X, Y) \to \FIsoc(X, Y')$ is an equivalence of categories.
\end{lemma}
\begin{proof}
The original but unpublished reference is \cite[Th\'eor\`eme 2.3.5]{berthelot-rigid}.
An alternate reference is \cite[Theorem~7.1.8]{lestum}.
\end{proof}

\begin{defn}
We define the category $\FIsoc^\dagger(X)$ of \emph{overconvergent $F$-isocrystals} on $X$ to be $\FIsoc(X,Y)$ for some (hence any, by Lemma~\ref{L:change compactification}) open immersion $X \to Y$ with $Y$ a proper $k$-variety. In particular, if $X$ itself is proper, then $\FIsoc^\dagger(X) = \FIsoc(X)$; in general, $\FIsoc^\dagger(X)$ is a stack for the Zariski and \'etale topologies on $X$.
\end{defn}

\begin{remark} \label{R:etale pushforward}
Retain notation as in Remark~\ref{R:pullback functoriality}.
If $X' \to X$ is finite \'etale of constant degree $d>0$ and $Y' \to Y$ is finite flat, 
one also obtains a pushforward functor $\FIsoc(X', Y') \to \FIsoc(X, Y)$ which multiplies ranks by $d$. In particular, if $X' \to X$ is finite \'etale, we obtain a pushforward functor $\FIsoc^\dagger(X') \to \FIsoc^\dagger(X)$.
\end{remark}

\begin{remark} \label{R:affine cover}
The pushforward functoriality of $\FIsoc$ is often used in conjunction with the following observation (a higher-dimensional analogue of Belyi's theorem in positive characteristic): 
any projective variety over $k$ of pure dimension $n$
admits a finite morphism to $\PP^n_k$ which is \'etale over $\AAA^n_k$
\cite{kedlaya-more-etale}.
Moreover, any given zero-dimensional subscheme of the smooth locus may be forced into the inverse image of $\AAA^n_k$;
in particular, the smooth locus is covered by open subsets which are finite \'etale over $\AAA^n_k$ (via various maps).
\end{remark}

\begin{remark} \label{R:closed point evaluation}
Let $\varphi: K \to K$ be the Witt vector Frobenius.
In case $X = \Spec k$, the categories $\FIsoc(X)$ and $\FIsoc^\dagger(X)$ coincide,
and may be described concretely as the category of finite-dimensional $K$-vector spaces
equipped with isomorphisms with their $\varphi$-pullbacks.

In general, choose any closed point $x \in X$ with residue field $\ell$ and put $L = \Frac W(\ell)$. Then the pullback functors $\FIsoc(X) \to \FIsoc(x), \FIsoc^\dagger(X) \to \FIsoc^\dagger(x)$ define fiber functors in $L$-vector spaces; however, these are not neutral fiber functors unless $\ell = \FF_p$.
For more on the Tannakian aspects of the categories $\FIsoc(X)$ and $\FIsoc^\dagger(X)$, see \cite{crew-mono};
for the special case of a finite base field, see also the discussion starting in \S\ref{sec:finite fields}.
\end{remark}

Much of the basic analysis of convergent and overconvergent $F$-isocrystals involves ``local models'' of the global statements under consideration. We describe the basic setup using notation as in \cite{dejong-barsotti}. 

\begin{remark}
Put $\Omega = W(k) \llbracket t \rrbracket$. Let $\Gamma$ be the $p$-adic completion of $W(k) ((t))$. Let $\Gamma_c$ be the subring of $\Gamma$ consisting of Laurent series
convergent in some region of the form $* \leq |t| < 1$. Each of these rings carries a Frobenius lift $\sigma$ with $\sigma(t) = t^p$ and a derivation $\frac{d}{dt}$. 

Define the categories
\[
\FIsoc(k \llbracket t \rrbracket), \FIsoc(k ((t))), \FIsoc^\dagger(k((t)))
\]
to consist of finite projective modules over the respective rings $\Omega[p^{-1}], \Gamma[p^{-1}], \Gamma_c[p^{-1}]$ equipped with compatible actions of $\sigma$ and $\frac{d}{dt}$. Here compatibility means that the commutation relation between $\sigma$ and $\frac{d}{dt}$ on the modules is the same as on the base ring:
\[
\frac{d}{dt} \circ \sigma = pt^{p-1} \sigma \circ \frac{d}{dt}.
\]

For some purposes, it is useful to consider also the ring $\calR$ consisting of the union of the rings of rigid analytic functions over $K$ on annuli of the form $* \leq |t| < 1$
(commonly called the \emph{Robba ring} over $K$). 
Note that $\Gamma_c$ is the subring of $\calR$ consisting of Laurent series with coefficients in $W(k)$. Let $\calR^+$ be the subring of $\calR$ consisting of formal power series (i.e., with only nonnegative powers of $t$); this is the ring of rigid analytic functions on the open unit $t$-disc over $K$.

Define the categories
\[
\FIsoc^\ddagger(k \llbracket t \rrbracket), \FIsoc^\ddagger(k((t)))
\]
to consist of finite projective modules over the respective rings $\calR^+, \calR$ equipped with compatible actions of $\sigma$ and $\frac{d}{dt}$ (note that this use of $\ddagger$ is not standard notation). We then have faithful functors
\[
\xymatrix{
\FIsoc(k \llbracket t \rrbracket) \ar[r] \ar[d] & \FIsoc^\dagger(k((t))) \ar[r] \ar[d] & \FIsoc(k((t))) \\
\FIsoc^\ddagger(k \llbracket t \rrbracket) \ar[r] & \FIsoc^\ddagger(k((t))) &
}
\]
but no comparison between $\FIsoc(k((t)))$ and $\FIsoc^\ddagger(k((t)))$.
\end{remark}

\begin{remark} \label{R:no Frobenius}
One can also define convergent and overconvergent isocrystals without Frobenius structure (in both the global and local settings);
on these larger categories, the fiber functors described in Remark~\ref{R:closed point evaluation} become neutral. This corresponds on the $\ell$-adic side to passing from representations of arithmetic fundamental groups to representations of geometric fundamental groups. 
However, there are some subtleties hidden in the construction: one must include an additional condition on the convergence of the formal Taylor isomorphism (which is forced by the existence of a Frobenius structure).
\end{remark}

\begin{remark}
One can also define convergent and overconvergent isocrystals (with or without Frobenius structure)
on nonsmooth varieties. For $X$ affine, this is done by choosing a smooth affine variety $Y$ containing $X$ as a closed subscheme, lift to a smooth formal scheme, and work on the inverse image of $X$ in the generic fiber of the lift under the specialization morphism
(the so-called \emph{tube} of $X$) in the convergent case, or some strict neighborhood thereof in the overconvergent case. See again \cite{berthelot-rigid} or \cite{lestum}.
\end{remark}

\section{Slopes}

We next discuss a basic feature of isocrystals admitting no $\ell$-adic analogue: the theory of \emph{slopes}. We begin with the situation at a point.

\begin{defn}
Let $r,s$ be integers with $s>0$ and $\gcd(r,s) = 1$. Let $\calF_{r/s} \in \FIsoc(k)$
be the object corresponding (via Remark~\ref{R:closed point evaluation}) to the $K$-vector space on the basis $\be_1,\dots,\be_s$ equipped with the $\varphi$-action
\[
\varphi(\be_1) = \be_2, \quad \dots, \quad \varphi(\be_{s-1}) = \be_s, \quad \varphi(\be_s) = p^r \be_1.
\]
One checks easily that 
\begin{equation} \label{eq:pure hom}
\Hom_{\FIsoc(k)}(\calF_{r/s}, \calF_{r'/s'}) = \begin{cases} D_{r,s} & r'/s' = r/s  \\
0 & r'/s' \neq r/s \end{cases}
\end{equation}
where $D_{r,s}$ denotes the division algebra over $K$ of degree $s$ and invariant $r/s$.
\end{defn}

\begin{theorem}[Dieudonn\'e--Manin] \label{T:DM}
Suppose that $k$ is algebraically closed.
Then every $\calE \in \FIsoc(\Spec k)$ is uniquely isomorphic to a direct sum
\[
\bigoplus_{r/s \in \QQ} \calE_{r/s}
\]
in which each factor $\calE_{r/s}$ is (not uniquely) isomorphic to a direct sum of copies of $\calF_{r/s}$. (Note that uniqueness is forced by \eqref{eq:pure hom}.)
\end{theorem}
\begin{proof}
This is the standard Dieudonn\'e--Manin classification theorem, the original reference for which is \cite{manin}. See also \cite[Theorem~14.6.3]{kedlaya-book} and
\cite{ding-ouyang}.
\end{proof}

\begin{defn}
For $\calE \in \FIsoc(k)$, choose an algebraic closure $\overline{k}$ of $k$
and let $\calE'$ be the pullback of $\calE$ to $\FIsoc(\overline{k})$. Then
the direct sum decomposition of $\calE$ given by
Theorem~\ref{T:DM} descends to $\calE$ (and is independent of the choice of $\overline{k}$).
We define the \emph{slope multiset} of $\calE$ to be the multisubset of $\QQ$ of cardinality equal to the rank of $\calE$ in which the multiplicity of $r/s$ equals $\rank \calE_{r/s}$;
the slope multiset is additive in short exact sequences \cite[Lemma~1.3.4]{katz-slope}.
We arrange the elements of the slope multiset into a convex Newton polygon with left endpoint $(0,0)$, called the \emph{slope polygon} of $\calE$. Note that the vertices of the slope polygon belong to $[0,\rank(\calE)] \times \ZZ$.

For $\calE \in \FIsoc(X)$, we define the \emph{slope multiset} and \emph{slope polygon}
of $\calE$ at $x \in X$ by pullback to $\Spec \kappa(x)^{\perf}$.
We say that $\calE$ is \emph{isoclinic} if the slope multisets at all points are equal to a single repeated value; if that value is 0, we also say that $\calE$ is \emph{unit-root} or \emph{\'etale}. By \eqref{eq:pure hom}, there are no nonzero morphisms between isoclinic objects of distinct slopes.
Moreover, the isoclinic and unit-root properties are preserved by formation of subquotients and extensions (between objects of the same slope).
\end{defn}

\begin{remark} \label{R:cyclic vector}
Since the action of Frobenius on an object of $\FIsoc(k)$ can be characterized by writing down the matrix of action on a single basis, one might wonder whether the Newton polygon of the characteristic polynomial of said matrix coincides with the slope polygon.
In general this is false; see \cite[\S 1.3]{katz-slope} for a counterexample. However,
it does hold when the basis is the one derived from a \emph{cyclic vector} for the action of Frobenius \cite[Lemma~5.2.4]{kedlaya-revisited}, i.e., when the matrix is the \emph{companion matrix} associated to its characteristic polynomial.
\end{remark}

\begin{remark} \label{R:rank 1}
Every $\calE \in \FIsoc(X)$ of rank 1 is isoclinic of some integer slope; this can either be proved directly or deduced from Theorem~\ref{T:polygon variation} below.
\end{remark}

\begin{remark} \label{R:sign convention}
The sign convention for slopes used here is the one from \cite{katz-slope}. 
However, in certain related contexts it is more natural to use the opposite sign convention. For example, in the theory of $\varphi$-modules over the Robba ring,
the sign convention taken here is used in \cite{kedlaya-locmono}; however,
this theory can be reformulated in terms of vector bundles on curves \cite{fargues, fargues-fontaine-durham, fargues-fontaine} and the opposite sign convention is the one consistent with geometric invariant theory. 
\end{remark}

Using slopes, we can now articulate two results that explain the relationship between \'etale $\QQ_p$-local systems and isocrystals. The first result says that in a sense, there are ``too few'' \'etale $\QQ_p$-local systems for them to serve as a good category of coefficient objects.

\begin{theorem}[Katz, Crew] \label{T:unit root1}
The category of unit-root objects in $\FIsoc(X)$ is equivalent to the category of \'etale $\QQ_p$-local systems on $X$. In particular, if $X$ is connected, this category is equivalent to the category of continuous representations of $\pi_1(X, \overline{x})$
on finite-dimensional $\QQ_p$-vector spaces (for any geometric point $\overline{x}$ of $X$).
\end{theorem}
\begin{proof}
See \cite[Theorem~2.1]{crew-f}.
\end{proof}

The second result says that on the other hand, there are also ``too many'' \'etale $\QQ_p$-local systems for them to serve as a good category of coefficient objects.

\begin{defn}
An \'etale $\QQ_p$-local system $V$ on $X$ is \emph{unramified} if 
the corresponding representations of the \'etale fundamental groups of the connected components of $X$ restrict trivially to all inertia groups. If $X$ admits an open immersion into a smooth proper variety $\overline{X}$, then by Zariski--Nagata purity, $V$ is unramified if and only if $V$ extends (necessarily uniquely) to an \'etale $\QQ_p$-local system on $\overline{X}$.
We say $V$ is \emph{potentially unramified} if there exists a finite \'etale cover $X' \to X$ such that the pullback of $V$ to $X'$ is unramified.
\end{defn}

\begin{theorem}[Tsuzuki] \label{T:unit root2}
In the equivalence of Theorem~\ref{T:unit root1}, the unit-root objects in $\FIsoc^\dagger(X)$ form a full subcategory of $\FIsoc(X)$ corresponding to the category of potentially unramified \'etale $\QQ_p$-local systems on $X$.
\end{theorem}
\begin{proof}
In the case $\dim X = 1$, this is \cite[Theorem~4.2.6]{tsuzuki-finite}.
For the general case, see \cite[Theorem~1.3.1, Remark~7.3.1]{tsuzuki-finite2}.
\end{proof}

\begin{remark} \label{R:unit root1 local model}
The local model of Theorem~\ref{T:unit root1} is that the category of unit-root objects in $\FIsoc(k((t)))$ is equivalent to the category of continuous representations of the absolute Galois group $G_{k((t))}$ on finite-dimensional $\QQ_p$-vector spaces.
The local model of Theorem~\ref{T:unit root2} is that the unit-root objects in
$\FIsoc^\dagger(k((t)))$ constitute the full subcategory in $\FIsoc(k((t)))$ 
corresponding to the representations with finite image of inertia.
See \cite[Theorem~4.2.6]{tsuzuki-finite} for discussion of both statements.
\end{remark}

\begin{remark} \label{R:unit root 1 plus 2}
By arguing as in \cite{tsuzuki-finite2},
one may prove a common generalization of Theorem~\ref{T:unit root1} and Theorem~\ref{T:unit root2}: for $X \to Y$ an open immersion,
the unit-root objects in $\FIsoc(X,Y)$ form a full subcategory corresponding to the category of \'etale $\QQ_p$-local systems $V$ on $X$ with the following property: there exists some proper morphism $Y' \to Y$ such that $X' = X \times_Y Y'$ is finite \'etale over $X$
and the pullback of $V$ to $X'$ extends to an \'etale $\QQ_p$-local system on $Y'$.
\end{remark}

We now consider the variation of the slope polygon over $X$.
\begin{theorem}[Grothendieck, Katz, de Jong--Oort, Yang] \label{T:polygon variation}
For $\calE \in \FIsoc(X)$, the following statements hold.
\begin{enumerate}
\item[(a)]
The slope polygon of $\calE$ is an upper semicontinuous function of $X$; moreover, its right endpoint is locally constant.
\item[(b)]
The locus of points where the slope polygon does not coincide with its generic value (which by (a) is Zariski closed) is of pure codimension $1$ in $X$.
\item[(c)]
Let $U$ be an open neighborhood of a point $x \in X$.
Suppose that the closure $Z$ of $x$ in $U$ has codimension at least $2$ in $U$.
If the slope polygons of $\calE$ at all points of $U \setminus Z$ share a common vertex, 
then this vertex also occurs in the slope polygon of $\calE$ at $x$.
(Beware that this statement does not apply to points of the slope polygon other than vertices.)
\end{enumerate}
\end{theorem}
\begin{proof}
Suppose first that $\calE$ arises from a crystal of finite \emph{locally free} $\calO_{X,\mathrm{crys}}$-modules via Theorem~\ref{T:isogeny category}. In this case,
we may deduce (a) from \cite[Theorem~2.3.1]{katz-slope}, (b) from
\cite[Theorem~4.1]{dejong-oort} or \cite[Main Theorem~1.6]{vasiu},
and (c) from \cite[Theorem~1.1]{yang}.

In light of Remark~\ref{R:not locally free}, this argument is not sufficient except when $\dim(X) = 1$. To proceed further, we may assume that $X$ is irreducible with generic point $\eta$.
To recover (a), we argue by noetherian induction.
By discarding a suitable closed subspace of codimension at least 2, we may deduce that there exists an open dense subscheme $U$ of $X$ on which the the slope polygon coincides with its value at $\eta$
(compare \cite[Lemma~2.5.1]{crew-f}). By restricting to curves in $X$, we may deduce that the slope polygon at every point lies on or above the value at $\eta$. Consequently, for each irreducible component $Z$ of $X \setminus U$,
the set of points $z \in Z$ at which the slope polygon of $\calE$ coincides with its value at $\eta$ is either empty or an open dense subscheme; in either case, its complement is closed in $Z$ and hence in $X$. 

Unfortunately, it is not clear how to use a similar approach to reduce (b) or (c) to the case of locally free crystals. We thus adopt a totally different approach; see Remark~\ref{R:purity of stratification}.
\end{proof}

\begin{remark} \label{R:specialization local model}
The reference given for Theorem~\ref{T:polygon variation}(a) also implies the local model statement: for
$\calE \in \FIsoc(k \llbracket t \rrbracket)$, the slope polygon of the pullback of $\calE$ to $\FIsoc(k)$ (the \emph{special slope polygon}) lies on or above the slope polygon of the pullback of $\calE$ to
$\FIsoc(k((t)))$ (the \emph{generic slope polygon}), with the same endpoint.
This statement can be generalized to $\calE \in \FIsoc^\dagger(k((t)))$ using slope filtrations
in $\FIsoc^\ddagger(k((t)))$; see Remark~\ref{R:slope filtration Robba}.
\end{remark}

In certain cases, the geometric structure on $X$ precludes the existence of nontrivial variation of slope polygons, as  in the following recent result of Tsuzuki (given a different proof by D'Addezio).
\begin{theorem}[Tsuzuki, D'Addezio]
For $k$ finite, $X$ an abelian variety over $k$, and $\calE \in \FIsoc(X)$, the slope polygon of $\calE$ is constant on $X$.
\end{theorem}
\begin{proof}
See \cite[Theorem~1.4]{tsuzuki-np} or \cite[Theorem~1.1]{daddezio-slopes}.
\end{proof}

\section{Slope filtrations}

We continue the discussion of slopes by considering filtrations by slopes. Such filtrations are loosely analogous to the filtration occurring in the definition of a variation of Hodge structures.

\begin{theorem}[after Katz] \label{T:filtration}
Suppose $\calE \in \FIsoc(X)$ has the property that the point $(m,n) \in \ZZ^2$ is a vertex of the slope polygon at every point of $\calE$.
Then there exists a short exact sequence
\[
0 \to \calE_1 \to \calE \to \calE_2 \to 0
\]
in $\FIsoc(X)$ with $\rank \calE_1 = m$ such that for each $x \in X$, the slope polygon of $\calE_1$ is the portion of the slope polygon of $\calE$ from $(0,0)$ to $(m,n)$.
\end{theorem}
\begin{proof}
In the case where $X$ is a curve, we may apply \cite[Corollary~2.6.2]{katz-slope}.
For general $X$, in light of Remark~\ref{R:not locally free} we may execute the same argument to obtain the desired exact sequence over some open dense subspace $U$ of $X$ with $\codim(X-U, X) \geq 2$. We may then conclude using
Zariski--Nagata purity (see Theorem~\ref{T:purity} and Remark~\ref{R:purity of stratification} below).
\end{proof}
\begin{cor}[after Katz] \label{C:filtration}
Suppose $\calE \in \FIsoc(X)$ has the property that the slope polygon of $\calE$ is constant on $X$.
Then $\calE$ admits a unique filtration
\[
0 = \calE_0 \subset \cdots \subset \calE_l = \calE
\]
such that each successive quotient $\calE_i/\calE_{i-1}$ is everywhere isoclinic of some slope $s_i$,
and $s_1 < \cdots < s_l$. We call this the \emph{slope filtration} of $\calE$.
\end{cor}

\begin{remark}
In Theorem~\ref{T:filtration}, it is not enough to assume that $(m,n)$ lies on the slope polygon at every point of $\calE$, even if one also assumes that $(m,n)$ is a vertex at each generic point of $X$.
\end{remark}

\begin{remark} \label{R:filtration local model}
One local model of Corollary~\ref{C:filtration} is that every object of $\FIsoc(k((t)))$ has a slope filtration \cite[Proposition~5.10]{kedlaya-locmono}.
A more substantial version is that for $\calE \in \FIsoc(k \llbracket t \rrbracket)$,
if the generic and special slope polygons coincide,
then $\calE$ admits a slope filtration \cite[Corollary~2.6.3]{katz-slope}.
A similar statement holds for $\calE \in \FIsoc^\dagger(k((t)))$; see Remark~\ref{R:slope filtration Robba}.
\end{remark}

\begin{remark} \label{R:filtration no connection}
The arguments in \cite{katz-slope} involve a finite projective module equipped only with a Frobenius action (and not an integrable connection).
On one hand, this means that
Theorem~\ref{T:filtration} remains valid in this setting, as does its local model
(Remark~\ref{R:filtration local model}). On the other hand,
to obtain Theorem~\ref{T:filtration} (or Remark~\ref{R:filtration local model}) as stated,
one must make an extra argument to verify that the filtration is respected also by the connection. To wit, the Kodaira--Spencer construction defines a morphism
$\calE_1 \to \calE_2$ of $\sigma$-modules which vanishes if and only if $\calE_1$ is 
stable under the connection; however, this vanishing is provided by \eqref{eq:pure hom}.
\end{remark}

There is no analogue of Theorem~\ref{T:filtration} for overconvergent $F$-isocrystals. 
Here is an explicit example.
\begin{example} \label{E:elliptic unit-root}
Let $X$ be the modular curve $X(N)$ for some $N \geq 3$ not divisible by $p$
(taking $N \geq 3$ forces this to be a scheme rather than a Deligne--Mumford stack).
Then the first crystalline cohomology of the universal elliptic curve over $X$ gives rise to an object $\calE$ of
$\FIsoc^\dagger(X)$ of rank 2. The slope polygon of $\calE$ generically has slopes $0,1$,
but there is a finite set $Z \subset X$ (the \emph{supersingular locus}) 
at which the slope polygon jumps to $1/2, 1/2$. Let $U$ be the complement of $Z$ in $X$ (the \emph{ordinary locus}); by Theorem~\ref{T:filtration}, the restriction of $\calE$ to
$\FIsoc(U)$ admits a rank 1 subobject which is unit-root.
However, no such subobject exists in $\FIsoc^\dagger(U)$; see Remark~\ref{R:subobjects}.

By completing at a supersingular point, we also obtain
an irreducible object of $\FIsoc(k\llbracket t \rrbracket)$ which remains irreducible in $\FIsoc^\dagger(k((t)))$ but not in $\FIsoc(k((t)))$.
\end{example}

\begin{remark} \label{R:reverse filtration}
Notwithstanding Example~\ref{E:elliptic unit-root}, one can formulate something like a filtration theorem for overconvergent $F$-isocrystals, at the expense of working in a ``perfect'' setting where the Frobenius lift is a bijection; since one cannot differentiate in such a setting, one only gets statements about individual liftings.

For simplicity, we discuss only the local model situation here.
Put $\Gamma^{\perf} = W(k((t))^{\perf})$; there is a natural Frobenius-equivariant embedding $\Gamma \to \Gamma^{\perf}$ taking $t$ to the Teichmuller lift $[t]$
(that is, the Frobenius lift $\sigma$ on $\Gamma$ corresponds to the unique Frobenius lift $\varphi$ on $\Gamma^{\perf}$).
Each element of $\Gamma^{\perf}$ can be written uniquely as a $p$-adically convergent series
$\sum_{n=0}^\infty p^n [\overline{x}_n]$ for some $\overline{x}_n \in k((t))^{\perf}$; let $\Gamma^{\perf}_c$ be the subset of $\Gamma^{\perf}$ consisting of those series for which the $t$-adic valuations of $\overline{x}_n$ are bounded below by some linear function of $n$ (for $n>0$).
One verifies easily that $\Gamma^{\perf}_c$ is a $\varphi$-stable subring of $\Gamma^{\perf}$ containing the image of $\Gamma_c$.

Suppose now that $\calE$ is a finite projective module over $\Gamma^{\perf}_c[p^{-1}]$ equipped with an isomorphism $\varphi^* \calE \cong \calE$. Using an argument of de Jong
\cite[Proposition~5.5]{dejong-barsotti}, one can show \cite[Proposition~5.11]{kedlaya-locmono}
that $\calE$ admits a unique filtration
implies that $\calE$ admits a unique filtration
\[
0 = \calE_0 \subset \cdots \subset \calE_l = \calE
\]
by $\varphi$-stable submodules
such that each successive quotient $\calE_i/\calE_{i-1}$ is everywhere isoclinic of some slope $s_i$,
and $s_1 > \cdots > s_l$. We call this the \emph{reverse slope filtration} of $\calE$.
\end{remark}

We add some additional remarks concerning the local situation.
\begin{remark} \label{R:slope filtration Robba plus}
For $\calE \in \FIsoc^\ddagger(k \llbracket t \rrbracket)$, an argument of Dwork
\cite[Lemma~6.3]{dejong-barsotti}
implies that $\calE$ admits a unique filtration specializing to the slope filtration in $\FIsoc(k)$, and that each subquotient descends uniquely to an isoclinic object in
$\FIsoc(k \llbracket t \rrbracket)$. In particular, the image of $\calE$ in $\FIsoc^\ddagger(k((t)))$ admits a filtration that in a certain sense reflects the special slope polygon of $\calE$. This sense is made more precise in Remark~\ref{R:slope filtration Robba} below.
\end{remark}

\begin{remark} \label{R:isoclinic Robba}
The functor from $\FIsoc^\dagger(k((t)))$ to $\FIsoc^\ddagger(k((t)))$ is not fully faithful in general, but it is fully faithful on the category of isoclinic objects of any fixed slope \cite[Theorem~6.3.3(b)]{kedlaya-revisited}. We declare an object of $\FIsoc^\ddagger(k((t)))$ to be \emph{isoclinic} of a particular slope if it arises from an isoclinic object of $\FIsoc^\dagger(k((t)))$ of that slope. 

Beware that the analogue of \eqref{eq:pure hom} in this context only holds when $r/s \leq r'/s'$. More precisely, if $\calE_1, \calE_2 \in \FIsoc^\ddagger(k((t)))$ are isoclinic of 
slopes $s_1, s_2$, then $\Hom_{\FIsoc^\ddagger(k((t)))}(\calE_1, \calE_2)$
vanishes when $s_1 < s_2$ (by \cite[Proposition~3.3.4]{kedlaya-revisited}),
equals the corresponding Hom-set in $\FIsoc^\dagger(k((t)))$ if $s_1 = s_2$ (by the full faithfulness statement quoted above), and is hard to control if $s_1 > s_2$.
\end{remark}

\begin{remark} \label{R:slope filtration Robba}
In light of Remark~\ref{R:slope filtration Robba}, one may ask whether
an arbitrary object $\calE \in \FIsoc^\ddagger(k((t)))$ admits a slope filtration in the sense of Corollary~\ref{C:filtration}. Such a filtration, were it to exist, would be unique by virtue of Remark~\ref{R:isoclinic Robba}; namely, under the geometric sign convention (Remark~\ref{R:sign convention}), it would coincide with the \emph{Harder--Narasimhan filtration} by destabilizing subobjects.
However, constructing such a filtration is made difficult by the fact that in this setting, it cannot be studied using cyclic vectors (as in Remark~\ref{R:cyclic vector}).
Nonetheless, with some effort one can prove existence of such a filtration \cite[Theorem~6.10]{kedlaya-locmono} (again using the Kodaira--Spencer argument
to pass from a filtration of $\sigma$-modules to a filtration of isocrystals)
and then use it to \emph{define} the slope polygon of $\calE$.
(For alternate expositions of the construction, see \cite[Theorem~6.4.1]{kedlaya-revisited}, \cite[Theorem~1.7.1]{kedlaya-rel}.)

For $\calE \in \FIsoc^\dagger(k((t)))$, one can now associate two slope polygons to $\calE$:
one arising from the image in $\FIsoc(k((t)))$, called the \emph{generic slope polygon};
and one arising from the image in $\FIsoc^\ddagger(k((t)))$, called the \emph{special slope polygon}. In case $\calE$ arises from $\FIsoc(k \llbracket t \rrbracket)$, these definitions agree with the ones from Remark~\ref{R:specialization local model}. 
One can make an \emph{extended Robba ring} containing both $\Gamma^{\perf}_c$ and
$\calR$ and use it to compare the slope filtration described above with the reverse
slope filtration (Remark~\ref{R:reverse filtration}), so as to obtain analogues of
Remark~\ref{R:specialization local model} and Remark~\ref{R:filtration local model}:
the special slope polygon again lies on or above the generic slope polygon, with the same right endpoint \cite[Proposition~5.5.1]{kedlaya-revisited},
and equality implies the existence of a slope filtration of $\calE$ itself \cite[Theorem~5.5.2]{kedlaya-revisited}.
\end{remark}

\begin{remark} \label{R:Crew1}
By combining Remark~\ref{R:unit root1 local model} with Remark~\ref{R:slope filtration Robba}, one sees that every object $\calE \in \FIsoc^\ddagger(k((t)))$
admits a filtration with the property that for some finite
\'etale morphism $\Spec k'((u)) \to \Spec k((t))$, the pullback to 
$\FIsoc^\ddagger(k'((u)))$ of each subquotient of the filtration is itself an object arising by pullback from $\FIsoc(k')$. (Technical note: forming the pullback involves changing Frobenius lifts, which is achieved using the Taylor isomorphism provided by the connection.) This is a statement formulated\footnote{Crew's exact wording was: ``It seems reasonable that \emph{any} overconvergent $F$-isocrystal on a smooth curve is quasi-unipotent.'' This was later interpreted as the first formal statement of Crew's conjecture.} by Crew \cite[\S 10.1]{crew-finite}, commonly known thereafter as \emph{Crew's conjecture}; the approach to Crew's conjecture we have just described is the one given in \cite{kedlaya-locmono}. Independent contemporaneous proofs were given by 
Andr\'e \cite{andre-mono} and Mebkhout \cite{mebkhout-mono} based on the theory of $p$-adic differential equations; see \cite[Theorem~20.1.4]{kedlaya-book} for a similar argument.
\end{remark}

\begin{remark} \label{R:local mono rep}
Let $X$ be a curve, let $x \in X$ be a closed point of residue field $k$, let $U$ be the complement of $x$ in $X$, and identify the completed local ring of $X$ at $x$ with $k \llbracket t \rrbracket$. For $\calE \in \FIsoc(U,X)$,
by applying Remark~\ref{R:Crew1} to the pullback of $\calE$ to $\FIsoc^\ddagger(k((t)))$,
we obtain a representation of $G_{k((t))}$ with finite image of inertia.
This is called the \emph{local monodromy representation} of $\calE$ at $x$, because it plays a similar role to that played in $\ell$-adic cohomology to the pullback of a local system from $X$ to $\Spec k((t))$; see Remark~\ref{R:Crew conjecture} for more details. For this reason, Crew's conjecture is also called the \emph{$p$-adic local monodromy theorem}; however, in the $p$-adic setting there is no natural definition of a \emph{global} monodromy representation which specializes to the local ones.
(See \cite{kedlaya-overview} for a careful construction of local monodromy representations.)
\end{remark}

\section{Restriction functors}
\label{sec:restriction}

Throughout \S\ref{sec:restriction}, let $X \to Y$ be an open immersion of $k$-varieties
(with no smoothness condition on $Y$),
let $U$ be an open dense subscheme of $X$,
and let $W$ be an open subscheme of $Y$ containing $U$.
We exhibit some properties of the restriction functor $\FIsoc(X,Y) \to \FIsoc(U,W)$; in the case of unit-root isocrystals, most of these statements can be predicted from 
Theorem~\ref{T:unit root1} and Theorem~\ref{T:unit root2}, but the proofs require additional ideas. In a few cases, the predictions turn out to be misleading.

We begin with an analogue of Zariski--Nagata purity (which has no local model).
In the unit-root case, this may be deduced from Remark~\ref{R:unit root 1 plus 2}.
\begin{theorem}[Kedlaya, Shiho] \label{T:purity}
Suppose that $\codim(X-U, X) \geq 2$.
\begin{enumerate}
\item[(a)]
The functor  $\FIsoc(X, Y) \to \FIsoc(U, Y)$
is an equivalence of categories. 
\item[(b)]
Suppose further that $Y$ is smooth, $Y \setminus X$ is a normal crossings divisor, and
$\codim(Y-W, Y) \geq 2$. Then the functor 
$\FIsoc(X, Y) \to \FIsoc(U, W)$
is an equivalence of categories.
\item[(c)]
The functors 
\[
\FIsoc^\dagger(X) \to \FIsoc^\dagger(U), \qquad \FIsoc(X) \to \FIsoc(U)
\]
are equivalences of categories.
\end{enumerate}
\end{theorem}
\begin{proof}
For (a), see \cite[Proposition~5.3.3]{kedlaya-semi1}.
For (b), see \cite[Theorem~3.1]{shiho-purity}.
For (c), apply (a) with $Y=X$ and (b) with $Y = X$, $W = U$.
\end{proof}

\begin{remark} \label{R:purity of stratification}
It should be pointed out that in Theorem~\ref{T:purity}, full faithfulness of the restriction morphism is quite elementary. For example, in the case $Y = X$ of part (a), full faithfulness reduces to the fact that 
with notation as in  Definition~\ref{D:convergent realization},
for $Q$ the open formal subscheme of $P$ supported on $U$, we have $H^0(Q_K, \calO) = H^0(P_K, \calO)$. This fact, whose proof we leave to the reader, might be thought of as a nonarchimedean analogue of the Hartogs theorem from the theory of several complex variables.

This weaker statement suffices for some important applications. For example,
to deduce Theorem~\ref{T:filtration} from the results of \cite{katz-slope},
as noted above one must extend the conclusion from $U$ to $X$ for some open dense subspace $U$ with $\codim(X-U, X) \geq 2$.
For this, by replacing $\calE$ with $\wedge^m \calE$ we may reduce to the case $m=1$;
in this case, $\calE_1$ is automatically a twist of a unit-root object (see Remark~\ref{R:rank 1}) and so corresponds to an \'etale $\QQ_p$-local system on $U$ by Theorem~\ref{T:unit root1}. The latter extends to $X$ by Zariski--Nagata purity  in the usual sense
\cite[Tag~0BMB]{stacks-project}; by Theorem~\ref{T:unit root1} again, this means that $\calE_1$ itself extends canonically to an object of $\FIsoc(X)$. 
We may thus apply full faithfulness in Theorem~\ref{T:purity} to conclude.

Similar considerations apply to part (c) (and therefore part (b)) of Theorem~\ref{T:polygon variation}(c), to give a proof which is completely independent of \cite{dejong-oort} and \cite{yang}. Namely, we may assume that $X$ is irreducible with generic point $\eta$.
Fix a vertex of the slope polygon of $\calE$ at $\eta$, and let $U$ be the subset of $X$
on which this vertex persists.
By Theorem~\ref{T:polygon variation}(a), $U$ is open; by Corollary~\ref{C:filtration},
the restriction of $\calE$ to $U$ admits a slope filtration.
If $\codim(X-U, X) \geq 2$, then by full faithfulness in Theorem~\ref{T:purity} this filtration extends over $X$;
this proves the claim.
\end{remark}

We continue with a general statement about restriction functors, which combines work of several authors; in addition to the results cited in the proof, see Remark~\ref{R:fully faithful1} and Remark~\ref{R:fully faithful local} for relevant attributions.
\begin{theorem}[de Jong, Kedlaya, Shiho] \label{T:fully faithful1}
The restriction functor
\[
\FIsoc(X,Y) \to \FIsoc(U, W)
\]
is fully faithful. In particular, the functors
\begin{gather*}
\FIsoc(X,Y) \to \FIsoc(X), \quad \FIsoc^\dagger(X) \to \FIsoc(X),  \\
\FIsoc(X, Y) \to \FIsoc(U, Y), \quad 
\FIsoc(X) \to \FIsoc(U), \quad
\FIsoc^\dagger(X) \to \FIsoc^\dagger(U)
\end{gather*}
are fully faithful.
\end{theorem}
\begin{proof}
By forming the composition
\[
\FIsoc(X,Y) \to \FIsoc(U,W) \to \FIsoc(U),
\]
we immediately reduce the general problem to the case $W = U$. In this case, the functor in question factors as
\[
\FIsoc(X, Y) \to \FIsoc(X) = \FIsoc(X,X) \to \FIsoc(U,X) \to \FIsoc(U).
\]
By \cite[Theorem~5.2.1]{kedlaya-semi1},
the functor $\FIsoc(X,X) \to \FIsoc(U,X)$ is fully faithful.
By \cite[Theorem~4.2.1]{kedlaya-semi2},
the functors $\FIsoc(X, Y) \to \FIsoc(X)$, $\FIsoc(U,X) \to \FIsoc(U)$
are fully faithful.
\end{proof}

\begin{remark} \label{R:fully faithful1}
For unit-root isocrystals, the full faithfulness of $\FIsoc^\dagger(X) \to \FIsoc(X)$
is included in Theorem~\ref{T:unit root2}; the general case is treated in \cite[Theorem~1.1]{kedlaya-full}. The proof of full faithfulness of $\FIsoc(X,Y) \to \FIsoc(X)$ appearing in \cite[Theorem~4.2.1]{kedlaya-semi2} is a small variant of the proof of \cite[Theorem~1.1]{kedlaya-full}; in particular, it involves reduction to the local model statement (Remark~\ref{R:fully faithful local}).

The full faithfulness of $\FIsoc^\dagger(X) \to \FIsoc^\dagger(U)$ follows from 
\cite[Th\'eor\`eme~4]{etesse}. The argument is extended in \cite[Theorem~5.2.1]{kedlaya-semi1} to obtain full faithfulness of $\FIsoc(X,Y) \to \FIsoc(U, Y)$; see also \cite{shiho-log} for some stronger results. 
\end{remark}

\begin{remark} \label{R:fully faithful local}
The local model of Theorem~\ref{T:fully faithful1} is the statement that the functors
\[
\FIsoc(k\llbracket t \rrbracket) \to \FIsoc^\dagger(k((t))),
\FIsoc^\dagger(k((t))) \to \FIsoc(k((t)))
\]
are fully faithful. The full faithfulness of the composite functor $\FIsoc(k \llbracket t \rrbracket) \to \FIsoc(k((t)))$ is due to de Jong \cite[Theorem~9.1]{dejong-barsotti}, and is the key ingredient in his proof of the analogue of Tate's extension theorem for $p$-divisible groups in equal positive characteristic. 
(See also \cite[Theorem~1.1]{kedlaya-frobmod} for a streamlined exposition.)

In fact, de Jong's approach is to first show that
$\FIsoc(k\llbracket t \rrbracket) \to \FIsoc^\dagger(k((t)))$ is fully faithful,
then to show that the restriction of 
$\FIsoc^\dagger(k((t))) \to \FIsoc(k((t)))$ to the essential image of $\FIsoc(k \llbracket t \rrbracket)$ is fully faithful. Both steps make essential use of the functor
$\FIsoc^\dagger(k((t))) \to \FIsoc^\ddagger(k((t)))$; for example, 
it is crucial that objects of $\FIsoc(k\llbracket t \rrbracket)$ admit slope filtrations in
$\FIsoc^\ddagger(k \llbracket t \rrbracket)$ (Remark~\ref{R:slope filtration Robba plus}).
The argument also makes essential use of the reverse slope filtration
(Remark~\ref{R:reverse filtration}).

Building on de Jong's approach, full faithfulness of
$\FIsoc^\dagger(k((t))) \to \FIsoc(k((t)))$ was established in \cite[Theorem~5.1]{kedlaya-full}.
The argument follows \cite{dejong-barsotti} fairly closely, except that Remark~\ref{R:slope filtration Robba plus} is replaced by Remark~\ref{R:slope filtration Robba} (see also Remark~\ref{R:Crew conjecture}).

Although this is not explained in \cite{dejong-barsotti}, one may use the results of that paper to establish full faithfulness of $\FIsoc(X) \to \FIsoc(U)$.
However, even if one does this, the argument still implicitly refers to $\FIsoc^\dagger(X)$; in fact, despite the fact that the statement can be formulated using only convergent $F$-isocrystals, we know of no proof that entirely avoids the use of overconvergent $F$-isocrystals.
\end{remark}

\begin{remark} \label{R:Liouville}
If one considers isocrystals without Frobenius structure,
then the analogue of full faithfulness for $\FIsoc^\dagger(X) \to \FIsoc^\dagger(U)$ holds
(by the same references as in Remark~\ref{R:fully faithful1}).
but the analogue of full faithfulness for $\FIsoc^\dagger(X)$ to $\FIsoc(X)$ fails
(see \cite{abe-full}).
The latter is related to known pathologies in the theory of $p$-adic differential equations related to $p$-adic Liouville numbers (i.e., $p$-adic integers which are overly well approximated by ordinary integers); see \cite{kedlaya-simons} for more discussion.
\end{remark}

\begin{remark}
An alternate approach to the full faithfulness problem for $\FIsoc^\dagger(X) \to \FIsoc(X)$, which does not go through the local model or depend on Crew's conjecture, is suggested by recent work of Ertl \cite{ertl} on an analogous problem in de Rham--Witt cohomology.
\end{remark}

On a related note, we mention the following results.
\begin{theorem}[Kedlaya]  \label{T:boundary glueing}
The functors
\begin{gather*}
\FIsoc(X,Y) \to \FIsoc(U,Y) \times_{\FIsoc(U)} \FIsoc(X) \\
\FIsoc^\dagger(X) \to \FIsoc^\dagger(U) \times_{\FIsoc(U)} \FIsoc(X)
\end{gather*}
are equivalences of categories.
\end{theorem}
\begin{proof}
See \cite[Proposition~5.3.7]{kedlaya-semi1}.
\end{proof}

\begin{cor} \label{C:boundary glueing1}
Set notation as in Remark~\ref{R:pullback functoriality} and suppose that $X' \to X$ is dominant and $Y' \to Y$ is surjective. Then the functors
\begin{gather*}
\FIsoc(X,Y) \to \FIsoc(X',Y') \times_{\FIsoc(X')} \FIsoc(X) \\
\FIsoc^\dagger(X) \to \FIsoc^\dagger(X') \times_{\FIsoc(X')} \FIsoc(X)
\end{gather*}
are equivalences of categories.
\end{cor}
\begin{proof}
By Theorem~\ref{T:fully faithful1}, the functor $\FIsoc(X,Y) \to \FIsoc(X)$ is fully faithful; there is thus no harm in replacing $X'$ with $X''$ for some morphism $X'' \to X'$.
In particular, we may reduce to the case where $X''$ is finite \'etale over some open dense subscheme of $X$.
Using Theorem~\ref{T:boundary glueing}, we may reduce further to the case where $X' \to X$ is finite.
By Lemma~\ref{L:change compactification}, we may also replace $Y$ with a blowup away from $X$;
using Gruson--Raynaud flattening \cite{gruson-raynaud}, we may further reduce to the case where
$Y' \to Y$ is finite flat (and surjective). In this case, if $f^* \calE \in \FIsoc(X')$ extends to $\calF \in \FIsoc(X',Y')$,
then using Remark~\ref{R:etale pushforward},
the restriction of $f_* \calF \in \FIsoc(X,Y)$ to $\FIsoc(X)$
has a summand isomorphic to $\calE$. By Theorem~\ref{T:fully faithful1},
the decomposition extends to a decomposition of $f_* \calF$ itself.
\end{proof}

\begin{remark} \label{R:boundary glueing}
In the case where $\dim(X) = 1$, Theorem~\ref{T:boundary glueing} admits a local variant: if $Y - X$ consists of a single $k$-rational point $x$, for $t$ a uniformizer of $Y$ at $x$, the functors
\begin{gather*}
\FIsoc(Y) \to \FIsoc(X,Y) \times_{\FIsoc(k((t)))} \FIsoc(k \llbracket t \rrbracket) \\
\FIsoc(X, Y) \to \FIsoc(X) \times_{\FIsoc(k((t)))} \FIsoc^\dagger(k((t)))
\end{gather*}
are equivalences.
\end{remark}

We next consider extension of subobjects.
\begin{theorem}[Kedlaya] \label{T:extend subobject}
Any subobject in $\FIsoc(U,Y)$ of an object of $\FIsoc(X,Y)$ extends to $\FIsoc(X,Y)$.
In particular, any subobject in $\FIsoc^\dagger(U)$ of an object of $\FIsoc^\dagger(X)$ extends to $\FIsoc^\dagger(X)$.
\end{theorem}
\begin{proof}
See \cite[Proposition~5.3.1]{kedlaya-semi1}.
\end{proof}

\begin{remark} \label{R:subobjects}
By contrast with Theorem~\ref{T:extend subobject}, not every subobject in $\FIsoc(X)$ of an object of $\FIsoc^\dagger(X)$ extends to $\FIsoc^\dagger(X)$. For example, set notation as in Example~\ref{E:elliptic unit-root}.
If the unit-root subobject of $\calE$ in $\FIsoc(U)$ could be extended to
$\FIsoc^\dagger(U)$, then by Theorem~\ref{T:fully faithful1}
and Theorem~\ref{T:boundary glueing} 
it would also extend to $\FIsoc^\dagger(X)$; this would imply that for any point $x \in X$ in the supersingular locus, the rigid cohomology of the elliptic curve corresponding to $x$ contains a distinguished line. However, using the endomorphism ring of such a curve (which is an order in a quaternion algebra over $\QQ$) one sees easily that no such distinguished line can exist.
\end{remark}

\begin{remark}
Given an exact sequence
\[
0 \to \calE_1 \to \calE \to \calE_2 \to 0
\]
with $\calE_1, \calE_2 \in \FIsoc^\dagger(X)$ and $\calE \in \FIsoc(X)$,
it does not follow that $\calE \in \FIsoc^\dagger(X)$;
for instance, this already fails in case $X = \AAA^1_K$ and $\calE_1, \calE_2$ are both the constant object in $\FIsoc^\dagger(X)$.
Similarly, if $\calE_1, \calE_2 \in \FIsoc^\dagger(X)$ and $\calE \in \FIsoc^\dagger(U)$, it does not follow that $\calE \in \FIsoc^\dagger(X)$ unless we allow for logarithmic structures (see Definition~\ref{D:log structure}).
\end{remark}

Although convergent subobjects of overconvergent $F$-isocrystals are in general not themselves overconvergent, they still seem to capture some structural information in the overconvergent category. 

\begin{remark}\label{R:optimistic conjecture}
In a previous version of this paper, we stated the following optimistic conjecture.
Let $\calE_1, \calE_2 \in \FIsoc^\dagger(X)$ be irreducible.
Let $\calF_1, \calF_2$ be objects in $\FIsoc(X)$ which are subobjects of $\calE_1, \calE_2$, respectively.
Then for every morphism $\calF_1 \to \calF_2$ in $\FIsoc(X)$, there exists a morphism $\calE_1 \to \calE_2$ in $\FIsoc^\dagger(X)$ such that the diagram
\[
\xymatrix{
\calF_1 \ar[r] \ar[d] & \calF_2 \ar[d] \\
\calE_1 \ar[r] & \calE_2
}
\]
commutes in $\FIsoc(X)$. This turns out to be false; see Example~\ref{exa:subobject morphisms} below.

However, we have no counterexample against the restricted form of the optimistic conjecture in which $\calF_i$ is an isoclinic subobject of $\calE_i$ of slope equal to the minimal generic slope of $\calE_i$.
Moreover, some partial results towards the restricted statement are known. 
\begin{itemize}
\item
Suppose that 
$\calE_1, \calE_2$ admit slope filtrations with respective first steps $\calF_1,\calF_2$. In this case, Tsuzuki \cite{tsuzuki-minimal} has proved this when either $X$ is a curve or $k$ is finite
(the second case reduces to the first via Theorem~\ref{T:cut irreducible} below).
Under the additional hypothesis that $\calE_1, \calE_2$ have all Frobenius slopes in the interval $[0,1]$,
an alternate proof has been given by D'Addezio (in preparation).

In particular, an irreducible overconvergent $F$-isocrystal with constant slope polygon is uniquely determined by the first step of its slope filtration.
Note that the condition on constant slope polygon is not essential: by Theorem~\ref{T:polygon variation} it holds on an open dense subspace $U$ of $X$, and by Theorem~\ref{T:fully faithful1} any morphism $\calE_1 \to \calE_2$ in
$\FIsoc^\dagger(U)$ lifts uniquely to a morphism in $\FIsoc^\dagger(X)$.

\item
Suppose that $X$ is irreducible, 
$\calF_1$ is the maximal subobject of $\calE_1$ in $\FIsoc(X)$ of minimal generic slope,
and $\calE_2 = \calF_2 = \calO_X$. The statement in this case is due to 
Ambrosi--D'Addezio \cite[Theorem~1.1.1]{ambrosi-daddezio}.
\end{itemize}
In addition, it should be possible to formulate other restricted forms of the optimistic conjecture, not contained in the previous one, for which one expects an affirmative answer; but it is not clear how to make the original formulation airtight. One option is Crew's \emph{parabolicity conjecture}, formulated in terms of monodromy groups
(see Definition~\ref{D:geometric monodromy}) and recently proved by D'Addezio \cite{daddezio2}.
\end{remark}

The following counterexample against the optimistic conjecture of Remark~\ref{R:optimistic conjecture} was provided by Marco D'Addezio.
\begin{example}[D'Addezio] \label{exa:subobject morphisms}
Retain notation as in Example~\ref{E:elliptic unit-root},
and let $\calF \in \FIsoc(U)$ be the unit-root subobject of $\calE$.
Define the objects
\[
\calE_1 = \Sym^2(\calE), \qquad \calE_2 = \wedge^2 \calE
\]
in $\FIsoc^\dagger(X)$ and the subobjects
\[
\calF_1 = \calE \otimes \calF \subset \calE_1, \qquad \calF_2 = \calE_2
\]
in $\FIsoc(U)$. Then there is a surjective morphism $\calF_1 \to \calF_2$ in $\FIsoc(U)$ given by
\[
\calF_1 = \calE \otimes \calF \to (\calE/\calF) \otimes \calF \cong \calF_2,
\]
but this cannot arise from a (necessarily surjective) morphism $\calE_1 \to \calE_2$ in $\FIsoc^\dagger(U)$
because $\calE_1$ is irreducible in $\FIsoc^\dagger(U)$
(see \cite[Proposition~4.11]{crew-mono} or Example~\ref{exa:monodromy}).
\end{example}

One can ask whether extendability of an $F$-isocrystal can be characterized on the level of curves (note that this question has no local model). 
Here is an example of such a statement. (It should be possible to remove the hypothesis on $k$ using
Poonen's finite field Bertini theorem \cite{poonen} or related results.)

\begin{theorem}[Shiho] \label{T:cut by curves}
The following statements hold.
\begin{enumerate}
\item[(a)]
An object of $\FIsoc(U,Y)$ extends to $\FIsoc(X,Y)$ if and only if for every curve $C \subseteq Y$, the pullback object in $\FIsoc(C \times_Y U, C)$
extends to $\FIsoc(C \times_U X, C)$.
\item[(b)]
An object of $\FIsoc^\dagger(U)$ lifts to $\FIsoc^\dagger(X)$ if and only if for every curve $C \subseteq X$,
the pullback object in $\FIsoc^\dagger(C \times_X U)$
lifts to $\FIsoc^\dagger(C)$.
\end{enumerate}
\end{theorem}
\begin{proof}
In the case where $k$ is uncountable, we obtain (a) by applying \cite[Theorem~0.1]{shiho-cut-extend}
(see the proof of Theorem~\ref{T:cut by curves log});
this immediately implies (b). (This part of the argument applies even in the absence of a Frobenius structure.)
In the general case, one may amend the argument as in the footnote to \cite[Lemma 2.4.13]{abe-companion}.
\end{proof}

It is reasonable to expect an analogue of Theorem~\ref{T:cut by curves} for extension from convergent to overconvergent isocrystals, but this is presently unknown.
Somewhat weaker results have been obtained by \cite{shiho-cut-over}; 
for instance, one must assume that the underlying connection extends to a strict neighborhood.
\begin{conj} \label{conj:cut over}
An object of $\FIsoc(X)$ extends to $\FIsoc(X,Y)$ if and only if for every curve $C \subseteq Y$, the pullback object in $\FIsoc(C \times_Y X)$ extends to $\FIsoc(C \times_Y X, C)$.
In particular,
an object of $\FIsoc(X)$ extends to $\FIsoc^\dagger(X)$ if and only if for every curve $C \subseteq X$, the pullback object in $\FIsoc(C)$ extends to $\FIsoc^\dagger(C)$.
(This holds for unit-root objects by Theorem~\ref{T:unit root1} and Theorem~\ref{T:unit root2}.)
\end{conj}

\begin{remark}
In conjunction with Theorem~\ref{T:cut by curves} (or more precisely, its expected extension to arbitrary $k$),
Conjecture~\ref{conj:cut over} would imply that
an object of $\FIsoc(U)$ extends to $\FIsoc(X)$ if and only if for every curve $C \subseteq X$, the pullback object in $\FIsoc(C \times_Y U)$ extends to $\FIsoc(C)$.
(Again, this holds for unit-root objects by Theorem~\ref{T:unit root1}.)
\end{remark}

One expects the following by analogy with Wiesend's theorem in the $\ell$-adic case
\cite{wiesend, drinfeld-deligne}, but we have no approach in mind except in the case where $k$ is finite.

\begin{conj} \label{conj:cut irreducible}
For $\calE \in \FIsoc^\dagger(X)$ irreducible, we can find a curve $C \subseteq X$ such that the pullback of $\calE$ to $\FIsoc^\dagger(C)$ is irreducible.
\end{conj}

\begin{remark}
In light of Remark~\ref{R:subobjects}, Conjecture~\ref{conj:cut irreducible} cannot be proved by reduction from $\FIsoc^\dagger(X)$ to $\FIsoc(X)$.
\end{remark}

\begin{theorem}[Abe-Esnault] \label{T:cut irreducible}
Conjecture~\ref{conj:cut irreducible} holds in case $k$ is finite and $\det(\calE)$ is of finite order.
\end{theorem}
\begin{proof}
See \cite[Theorem~0.3]{abe-esnault}.
\end{proof}

\begin{remark}
The proof of Theorem~\ref{T:cut irreducible} relies on the theory of weights
(\S\ref{sec:weights}) and the theory of companions (see \cite{kedlaya-companions}).
An alternate proof using these ingredients, but otherwise quite different in nature, will be given in \cite{kedlaya-companions}.
\end{remark}

\begin{remark}
It is possible for an object of $\FIsoc(X)$ to admit an overconvergent Frobenius structure with respect to
one particular lift of Frobenius without itself being an object of $\FIsoc^\dagger(X)$. For example,
it is possible to have $\calE \in \FIsoc^\dagger(X)$ with constant Newton polygon for which,
for a suitable choice of the Frobenius lift, the Frobenius action on $\calE$ induces
an overconvergent Frobenius structure on the steps of the slope filtration;
some explicit examples were found by Dwork \cite{dwork}.
\end{remark}

\section{Slope gaps}

We next study the behavior of gaps between slopes, starting with a cautionary remark.
\begin{remark}
Note that in general, a persistent gap between slopes is not enough to guarantee the existence of a slope filtration. That is, suppose that $\calE \in \FIsoc(X)$ has the property that for some positive integer $k < \rank(\calE)$, the $k$-th and $(k+1)$-st smallest slopes of $\calE$ at each point of $X$ are distinct.
Then $\calE$ need not admit a subobject of rank $k$ whose slopes at each point are precisely the $k$ smallest slopes of $\calE$ at that point. Namely, by Theorem~\ref{T:polygon variation},
this would imply that the sum of the $k$ smallest slopes is locally constant, which can fail in examples (see Example~\ref{exa:slope gap}).
However, this does hold if the gap is large enough; see Theorem~\ref{T:transversality}.
\end{remark}

\begin{example} \label{exa:slope gap} 
Let $Y$ be the moduli space of principally polarized abelian threefolds with full level $N$ structure for some $N \geq 3$ not divisible by $p$. Then the first crystalline cohomology of the universal abelian variety over $Y$
is an object $\calE$ of $\FIsoc^\dagger(Y)$ of rank 6. 
It is known (e.g., see \cite{chai-oort}) that the image of the slope polygon map for $\calE$ consists of all Newton polygons with nonnegative slopes and right endpoint $(6, 3)$.
In particular, we can find a curve $X$ in $Y$ such that the pullback of $\calE$ to $X$
has slopes $0,0,0,1,1,1$ at its generic point and $\frac{1}{3},\frac{1}{3},\frac{1}{3},\frac{2}{3},\frac{2}{3},\frac{2}{3}$ at some closed point. Since the smallest 3 slopes do not have constant sum, they cannot be isolated using a slope filtration.
\end{example}

Recall that there is a loose analogy between isocrystals and variations of Hodge structure.
With Griffiths transversality in mind, one may ask whether a persistent gap between slopes of length greater than 1 gives rise to a partial slope filtration. In fact, an even stronger statement holds: it is enough for such a gap to occur generically. 

\begin{theorem}[Drinfeld--Kedlaya] \label{T:transversality}
Suppose that $\calE \in \FIsoc(X)$ 
(resp.\ $\calE \in \FIsoc^\dagger(X)$)
has the property that for some positive integer $k$,
the difference between the $k$-th and $(k+1)$-st smallest slopes of $\calE$ at each generic point of $X$ is strictly greater than $1$.
\begin{enumerate}
\item[(a)]
At each $x \in X$, the sum of the $k$ smallest slopes of $\calE_x$ is equal to a locally constant value, and the difference between the $k$-th and $(k+1)$-st smallest slopes of $\calE$ is strictly greater than $1$.
\item[(b)]
There is a splitting $\calE \cong \calE_1 \oplus \calE_2$ of $\calE$ in $\FIsoc(X)$ 
(resp.\ $\FIsoc^\dagger(X)$) with $\rank(\calE_1) = k$ such that the slopes of $\calE_1$ at each point are exactly the $k$ smallest slopes of $\calE$ at that point.
\end{enumerate}
\end{theorem}
\begin{proof}
In light of Theorem~\ref{T:fully faithful1}, it is only necessary to prove Theorem~\ref{T:transversality} in the case $\calE \in \FIsoc(X)$. This is proved in \cite[Theorem~1.1.4]{drinfeld-kedlaya} using the Cartier operator;
see Lemma~\ref{L:transversality2} for a variant proof.
\end{proof}

\begin{remark}
Theorem~\ref{T:transversality} implies that if $X$ is irreducible and $\calE \in \FIsoc^\dagger(X)$ is indecomposable, then there is no gap of length greater than 1 between consecutive slopes of $\calE$ at the generic point of $X$. However, such gaps can occur at other points of $X$; see
\cite[Appendix]{drinfeld-kedlaya} for some examples.
\end{remark}

\begin{remark}
Theorem~\ref{T:transversality} can be used to obtain nontrivial consequences about the Newton polygons of Weil $\overline{\QQ}_{\ell}$-sheaves, refining results of V. Lafforgue \cite{lafforgue-hecke}. See \cite{drinfeld-kedlaya} for more discussion.
\end{remark}

\begin{remark}
In the overconvergent case, another approach to Theorem~\ref{T:transversality} has been given by Kramer-Miller
(in preparation). This avoids the dependence on the full faithfulness of restriction
(Theorem~\ref{T:fully faithful1})
but does introduce a dependence on Theorem~\ref{T:cut irreducible} which is not present in the approach of
\cite{drinfeld-kedlaya}.
\end{remark}

\section{Logarithmic compactifications}

As in other cohomology theories, a key technical tool in the study of overconvergent $F$-isocrystals on nonproper varieties is the formation of certain logarithmic compactifications.

\begin{defn} \label{D:log structure}
Suppose that $X \to \overline{X}$ is an open immersion with $\overline{X}$ smooth and $\overline{X} - X$ a normal crossings divisor. Let $\overline{X}_{\log}$ denote the scheme $\overline{X}$ equipped with the logarithmic structure defined by the divisor $\overline{X} - X$; 
one can then define the associated category $\FIsoc(\overline{X}_{\log})$ of convergent log-$F$-isocrystals \cite{shiho-crys1, shiho-crys2}.

To give a local description of this category, suppose that there exist a smooth affine formal scheme $P$ over $W(k)$ 
with $P_k \cong \overline{X}$, a relative normal crossings divisor $Z$ on $P$ with $Z_k \cong \overline{X} - X$, and a Frobenius lift $\sigma: P \to P$ which acts on $Z$.
Then an object of $\FIsoc(\overline{X}_{\log})$ may be viewed as a vector bundle $\calE$ on $P_K$ equipped with an integrable logarithmic connection (for the logarithmic structure defined by $Z_K$) and an isomorphism $\sigma^* \calE \to \calE$ of logarithmic $\mathcal{D}$-modules.
\end{defn}

\begin{defn}
Given an integrable logarithmic connection, the resulting map
$\calE \to \calE \otimes_{\calO_{P_K}} \Omega^{\log}_{P_K/K}/\Omega_{P_K/K}$
induces an $\calO_{Z_K}$-linear endomorphism of $\calE|_{Z_K}$ called the \emph{residue map}. The eigenvalues of the residue map must be killed by differentiation, and thus belong to $K$; the presence of the Frobenius structure forces the set of eigenvalues to be stable under multiplication by $p$. That is, any object of $\FIsoc(\overline{X}_{\log})$ has nilpotent residue map. Note that this would fail if we only required $\sigma^* \calE \to \calE$ to be an isomorphism away from $Z_K$; in this case, only the reductions modulo $\ZZ$ of the eigenvalues of the residue map would form a set stable under multiplication by $p$, so they would only be constrained to be rational numbers.
\end{defn}

\begin{theorem}[Kedlaya]
The functor $\FIsoc(\overline{X}_{\log}) \to \FIsoc(X, \overline{X})$ is fully faithful.
In particular, if $\overline{X}$ is proper, then $\FIsoc(\overline{X}_{\log}) \to \FIsoc^\dagger(X)$  is fully faithful.
\end{theorem}
\begin{proof}
See \cite[Theorem~6.4.5]{kedlaya-semi1}.
\end{proof}

Theorem~\ref{T:cut by curves} admits the following logarithmic analogue.
\begin{theorem}[Shiho] \label{T:cut by curves log}
An object of $\FIsoc^\dagger(X)$ extends to $\FIsoc(\overline{X}_{\log})$ if and only if
for every curve $C \subseteq \overline{X}$, the pullback object in $\FIsoc^\dagger(C \times_{\overline{X}} X)$ extends to $\FIsoc(C_{\log})$ (where the logarithmic structure on $C$ is the one pulled back from $\overline{X}$).
\end{theorem}
\begin{proof}
As in the proof of Theorem~\ref{T:cut by curves}, this ultimately follows from the proof of \cite[Theorem~0.1]{shiho-cut-extend}, taking the subset $\Sigma$ of $\ZZ_p^s$ therein to be identically zero. In that case, the condition of ``$\Sigma$-unipotent monodromy'' in \cite[Theorem~0.1]{shiho-cut-extend} corresponds to log-extendability
as per \cite[Proposition~6.3.2]{kedlaya-semi1}. One can recover Theorem~\ref{T:cut by curves} from this by noting that extendability in $\FIsoc(X, Y)$ is equivalent to
log-extendability plus vanishing of the residue along each boundary divisor (see \cite[Theorem~5.2.1]{kedlaya-semi1}), and the latter can be detected on any \emph{single} curve meeting that divisor.
\end{proof}

\begin{remark}
In light of Theorem~\ref{T:cut by curves log},  Conjecture~\ref{conj:cut over} would imply that
an object of $\FIsoc(X)$ extends to $\FIsoc(\overline{X}_{\log})$ if and only if
for every curve $C \subseteq \overline{X}$, the pullback object in $\FIsoc(C \times_{\overline{X}} X)$ extends to $\FIsoc(C_{\log})$.
\end{remark}

In general, not every object of $\FIsoc^\dagger(X)$ extends to $\FIsoc(\overline{X}_{\log})$. However, the obstruction to extending can always be eliminated using a finite cover of varieties. Note that the unit-root case of the following theorem is an immediate consequence of Theorem~\ref{T:unit root2}.

\begin{theorem}[Kedlaya] \label{T:semistable}
Given $\calE \in \FIsoc^\dagger(X)$, there exist an alteration $f: X' \to X$
in the sense of de Jong \cite{dejong-alterations}
and an open immersion $j: X' \to \overline{X}'$ with $\overline{X}'$ smooth proper
and $\overline{X}' - X'$ a normal crossings divisor, such that the pullback of $\calE$
to $\FIsoc^\dagger(X')$ extends to $\FIsoc(\overline{X}'_{\log})$.
\end{theorem}
\begin{proof}
For the case $\dim X = 1$, see \cite[Theorem~1.1]{kedlaya-semi-curve}.
For the general case, see \cite[Theorem~5.0.1]{kedlaya-semi4}.
\end{proof}

\begin{remark} \label{R:Crew conjecture}
The local model of Theorem~\ref{T:semistable} is the following statement: for any $\calE \in \FIsoc^\dagger(k((t)))$, there exists a finite \'etale morphism $\Spec k'((u)) \to \Spec k((t))$ such that the pullback of
$\calE$ to $\FIsoc^\dagger(k'((u)))$ extends to the category
$\FIsoc(k \llbracket u\rrbracket_{\log})$ of finite projective $W(k') \llbracket u \rrbracket[p^{-1}]$-modules
equipped with compatible actions of the Frobenius lift $u \mapsto u^p$ and the derivation $u \frac{d}{du}$.
This was stated formally by de Jong \cite[\S 5]{dejong-icm},
but was known to Crew to be a special case of his conjecture formulated in \cite[\S 10.1]{crew-finite};
more precisely, $\calE \in \FIsoc^\dagger(k((t)))$
lifts to $\FIsoc(k \llbracket t \rrbracket_{\log})$ if and only if its image in
$\FIsoc^\ddagger(k((t)))$ is a successive extension of objects, each of which arises by pullback from $\FIsoc(k)$.
In light of Remark~\ref{R:Crew1}, the resolution of Crew's conjecture thus yields the statement in question.
\end{remark}

\begin{remark}
In the case $\dim X = 1$, Theorem~\ref{T:semistable} is an easy consequence of the local model statement described in Remark~\ref{R:Crew conjecture}. In the wake of de Jong establishing his alterations theorem as a weak replacement for resolution of singularities in positive characteristic (Remark~\ref{R:alterations}), the general statement of Theorem~\ref{T:semistable} was formulated as a natural higher-dimensional analogue of the one-dimensional case; it first appeared (in a sentence of the form ``One can ask...'') in \cite[\S 5]{dejong-icm} and was formally conjectured by Shiho \cite[Conjecture~3.1.8]{shiho-crys2}.

The proof of Theorem~\ref{T:semistable} is the culmination of the sequence of papers
\cite{kedlaya-semi1, kedlaya-semi2, kedlaya-semi3, kedlaya-semi4}
(plus \cite[Appendix]{kedlaya-connections} for a crucial erratum to \cite{kedlaya-semi4})
where it is described as a \emph{semistable reduction theorem} for overconvergent $F$-isocrystals
(based on Crew's usage of the term \emph{semi-stable} in the one-dimensional case \cite[\S 10.1]{crew-finite}). 
The principal difficulty in the higher-dimensional case is that the alteration is generally forced to include some wildly ramified cover, whose singularities are hard to control; consequently, one cannot simply argue using Theorem~\ref{T:purity} and the one-dimensional case. Rather, one must work locally on the Riemann--Zariski space of the variety. Similar difficulties arise in trying to formulate a higher-dimensional analogue of the formal classification of meromorphic differential equations; see
\cite{kedlaya-goodformal1, kedlaya-goodformal2}.
\end{remark}

\begin{remark} \label{R:alterations}
Note that de Jong's alteration theorem is required even to produce the pair $X', \overline{X}'$ with the prescribed smoothness properties; the nature of de Jong's proof is such that one has very little control over the finite locus of the alteration. One might hope that under a strong hypothesis on resolution of singularities, Theorem~\ref{T:semistable} can be strengthened to ensure that the alteration $f$ is finite \'etale over $X$. This can be achieved when $\dim X = 1$: it is enough to ensure that $f$
trivializes the local monodromy representations (Remark~\ref{R:local mono rep}), which can be achieved via careful use of Katz--Gabber local-to-global extensions \cite{katz-local-to-global}. It is less clear whether one should even expect this to be possible when $\dim X > 1$, as there is in general no \emph{global} monodromy representation controlling the situation (compare Remark~\ref{R:local mono rep}).
However, using the theory of companions, modulo resolution of singularities this can be established when $k$ is finite \cite[Remark~4.4]{abe-esnault}: there exists a finite \'etale cover of $X$ which trivializes an $\ell$-adic companion modulo $\ell$ (for some prime $\ell \neq p$), and any alteration that factors through this cover suffices to achieve semistable reduction.
\end{remark}

\section{Cohomology}

Having studied the coefficient objects of rigid cohomology up to now, it is finally time to introduce the cohomology theory itself. Again, we fall back on \cite{lestum} for missing foundational discussion.

\begin{defn}
For $i \geq 0$ and $\calE \in \FIsoc^\dagger(X)$, 
let $H^i_{\rig}(X, \calE)$ denote the $i$-th \emph{rigid cohomology} group of $X$ with coefficients in $\calE$; it is a $K$-vector space equipped with an isomorphism with its $\varphi$-pullback.

One may describe rigid cohomology concretely in case $X$ is affine.
Let $P$ be a smooth affine formal scheme with $P_k \cong X$; then $\calE$ can be realized as a vector bundle with integrable connection on a strict neighborhood $U$ of $P_K$ in a suitable ambient space.
The rigid cohomology is then obtained by taking the hypercohomology of the de Rham complex
\[
0 \to \calE \stackrel{\nabla}{\to} \calE \otimes_{\calO_{U}} \Omega^1_{U/K}
\stackrel{\nabla}{\to} \calE \otimes_{\calO_{U}} \Omega^2_{U/K} \to \cdots,
\]
then taking the direct limit over (decreasing) strict neighborhoods.
For example, if $X = \AAA^n_k$, we may take $P$ to be the formal affine $n$-space, identify $P_K$ with the closed unit polydisc in $T_1,\dots,T_n$, then take the family of strict neighborhoods to be polydiscs of radii strictly greater than 1.
\end{defn}

\begin{remark}
For constant coefficients, the computation of rigid cohomology in the affine case agrees with the definition of ``formal cohomology'' by Monsky--Washnitzer \cite{monsky-washnitzer}, which was one of Berthelot's motivations for the definition of rigid cohomology. The key example is that of the affine line with constant coefficients: the de Rham complex over the closed unit disc has infinite-dimensional cohomology, whereas 
rigid cohomology behaves as one would expect from the Poincar\'e lemma (i.e., $H^0$ is one-dimensional and $H^1$ vanishes).
\end{remark}

\begin{theorem}[Ogus] \label{T:cohomology comparison}
Suppose that $X$ is smooth and proper, and let $\calE$ be the object of $\FIsoc(X) = \FIsoc^\dagger(X)$ corresponding to a crystal $M$ of finite $\calO_{X,\crys}$-modules via 
Theorem~\ref{T:isogeny category}. Then there are canonical isomorphisms
\[
H^i(X_{\crys}, M) \otimes_{\ZZ} \QQ \cong H^i_{\rig}(X, \calE) \qquad (i \geq 0).
\]
\end{theorem}
\begin{proof}
See \cite[Theorem~0.0.1]{ogus-topos}.
\end{proof}

\begin{theorem}[Kedlaya] \label{T:finite dimensional}
For $\calE \in \FIsoc^\dagger(X)$, the $K$-vector spaces $H^i_{\rig}(X, \calE)$ are finite-dimensional for all $i \geq 0$ and zero for all $i>2\dim X$.
\end{theorem}
\begin{proof}
See \cite[Theorem~1.2.1]{kedlaya-finiteness}. Alternatively, this can be deduced from
Theorem~\ref{T:semistable} using the fact that Theorem~\ref{T:cohomology comparison} can be extended to logarithmic isocrystals (see \cite{shiho-log}).
\end{proof}

\begin{remark}
Theorem~\ref{T:finite dimensional} fails for convergent $F$-isocrystals if $X$ is not proper: Theorem~\ref{T:cohomology comparison} (suitably stated) remains true without the properness condition, whereas crystalline cohomology for open varieties does not have good finiteness properties. More subtly, Theorem~\ref{T:finite dimensional} also fails for overconvergent isocrystals without Frobenius structure (Remark~\ref{R:no Frobenius}), due to issues involving $p$-adic Liouville numbers (see
Remark~\ref{R:Liouville}).
\end{remark}

For an overconvergent $F$-isocrystal on a curve,
we have the following analogue of the Grothendieck--Ogg--Shafarevich formula \cite{grothendieck-gos}.
The original formulation is due to Garnier \cite[Proposition~5.3.2]{garnier}, though it had to be made conditionally because Theorem~\ref{T:finite dimensional} was not available.

\begin{theorem}[Christol--Mebkhout, Crew, Matsuda, Tsuzuki]
Assume that $k$ is algebraically closed.
Suppose that $X$ is geometrically irreducible of dimension $1$, and let $\overline{X}$ be the smooth compactification of $X$. For $\calE \in \FIsoc^\dagger(X)$ and $x \in \overline{X} - X$, let $\Swan_x(\calE)$ denote the Swan conductor of the local monodromy representation of $\calE$ at $x$ (Remark~\ref{R:local mono rep}). Then
\[
\sum_{i=0}^2 (-1)^i \dim_K H^i_{\rig}(X, \calE)
= \chi(X) \rank(\calE) - \sum_{x \in \overline{X} - X} \Swan_x(\calE).
\]
\end{theorem}
\begin{proof}
See \cite[Theorem~4.3.1]{kedlaya-weil2}.
\end{proof}

\begin{theorem}
Rigid cohomology (of an overconvergent $F$-isocrystal) satisfies cohomological descent for proper hypercoverings.
\end{theorem}
\begin{proof}
See \cite[Corollary~2.2.3]{tsuzuki-descent}.
\end{proof}

\begin{remark}
There is also a theory of rigid cohomology with compact support admitting a form of Poincar\'e duality; see \cite{kedlaya-finiteness}. This is relevant for the Lefschetz trace formula; see Remark~\ref{R:trace formula}.
\end{remark}

\section{Finite fields}
\label{sec:finite fields}

We now specialize to the situation over finite fields. In order to best simulate the $\ell$-adic setting, we must 
promote the categories of isocrystals, which are linear over some finite extension of $\QQ_p$, to $\overline{\QQ}_p$-linear categories. We follow the general approach of \cite{abe-companion}.

\begin{hypothesis} \label{H:finite fields}
Throughout \S\ref{sec:finite fields}, assume that $k = \FF_q$ is finite
and choose a homomorphism $j: W(\overline{k}) \hookrightarrow \overline{\QQ}_p$.
For $n$ a positive integer, let $k_n$ be the degree $n$ subextension of $\overline{k}$ over $k$,
and put $X_n := X \times_k k_n$.
\end{hypothesis}

\begin{defn}
For each finite extension $L$ of $\QQ_p$ within $\overline{\QQ}_p$, 
let $\FIsoc(X) \otimes L$ (resp.\ $\FIsoc^\dagger(X) \otimes L$ be the category of objects of $\FIsoc(X)$ (resp.\ $\FIsoc^\dagger(X)$) equipped with a $\QQ_p$-linear action of $L$. Let
$\FIsoc(X) \otimes \overline{\QQ}_p$ (resp.\ 
$\FIsoc^\dagger(X) \otimes \overline{\QQ}_p$) be the 2-colimit of the categories
$\FIsoc(X) \otimes L$ (resp.\
$\FIsoc^\dagger(X) \otimes L$) over all finite extensions $L$ of $\QQ_p$ within $\overline{\QQ}_p$.

We extend the tensor product operation to $\FIsoc(X) \otimes \overline{\QQ}_p$ (and similarly 
$\FIsoc^\dagger(X) \otimes \overline{\QQ}_p$) as in \cite[\S 2.2]{abe-companion}. Given two objects
$\calE_1, \calE_2$ in $\FIsoc(X) \otimes L$ for some $L$, the tensor product
$\calE := \calE_1 \otimes \calE_2$ in $\FIsoc(X)$ inherits two distinct $L$-linear structures, and we define the tensor product in $\FIsoc(X) \otimes L$ to be the maximal quotient of $\calE$ on which the two $L$-linear structures coincide. Similarly, for each positive integer $n$,
applying the base extension functor $\FIsoc(X) \to \FIsoc(X_n)$ to the underlying object of some
$\calE \in \FIsoc(X) \otimes \overline{\QQ}_p$ yields an object
$\calE_n$ inheriting two distinct $W(k_n)$-linear structures
(one of them coming via $j$), and we define the base extension functor $\FIsoc(X) \otimes \overline{\QQ}_p \to
\FIsoc(X_n) \otimes \overline{\QQ}_p$ so as to take $\calE$ to the maximal quotient of $\calE_n$ on which the two
$W(k_n)$-linear structures coincide.

To justify the omission from $k$ in the notation, we observe that the category $\FIsoc(X) \otimes \overline{\QQ}_p$ remains unchanged if one changes the structure morphism $X \to \Spec k$ to $X \to \Spec k_n$. More precisely, if $X$ is irreducible and $k_n$ is the normalization of $k$ in $k(X)$, then the composition of the base extension functor 
from $\FIsoc(X) \otimes \overline{\QQ}_p$ (defined relative to $k$) to $\FIsoc(X_n) \otimes \overline{\QQ}_p$
(defined relative to $k_n$) with pullback from $X_n$ to one of its connected components is an equivalence.
See also Definition~\ref{D:linearized Frobenius action} for the case $X = \Spec(k_n)$.
\end{defn}

\begin{remark} \label{R:promotion1}
With the previous caveats about tensor products and base extensions,
all of the previous results about $\FIsoc(X)$ and $\FIsoc^\dagger(X)$ can be formally promoted to statements about $\FIsoc(X) \otimes \overline{\QQ}_p$ and $\FIsoc^\dagger(X) \otimes \overline{\QQ}_p$, which we will mostly use without further comment.
One cautionary remark: for objects in $\FIsoc(X) \otimes L$, the $y$-coordinates of the vertices of the slope polygon belong not to $\ZZ$ but to $e^{-1} \ZZ$ where $e$ is the absolute ramification index of $L$.
\end{remark}

We spell out explicitly one instance of Remark~\ref{R:promotion1}, corresponding to Theorems~\ref{T:unit root1}
and~\ref{T:unit root2}.
\begin{theorem} \label{T:promotion unit-root}
Let $L$ be a finite extension of $\QQ_p$.
\begin{enumerate}
\item[(a)]
The category of unit-root objects in $\FIsoc(X) \otimes L$ is equivalent to the category of \'etale $L$-local systems on $X$. 
\item[(b)]
Under this equivalence, the unit-root objects in $\FIsoc^\dagger(X) \otimes L$ correspond to the potentially unramified $L$-local systems on $X$.
\end{enumerate}
\end{theorem}
\begin{proof}
This follows by applying Theorems~\ref{T:unit root1}
and~\ref{T:unit root2} to the underlying objects in $\FIsoc(X)$ and $\FIsoc^\dagger(X)$ on one hand, and the underlying \'etale $\QQ_p$-local systems on the other hand.
\end{proof}

Over a finite field, we can define the $L$-function associated to an overconvergent $F$-isocrystal and formulate the Lefschetz trace formula for Frobenius.
\begin{defn} \label{D:linearized Frobenius action}
Suppose that $k = \FF_q$ is finite. For $n$ a positive integer, 
put $k_n = \FF_{q^n} \subseteq \overline{k}$ and $K_n = \Frac W(k_n)$.
An object of $\FIsoc^\dagger(k_n) \otimes \overline{\QQ}_p$
corresponds to a finite $(K_n \otimes_{\QQ_p} \overline{\QQ}_p)$-module equipped with an isomorphism with its $(\varphi \otimes 1)$-pullback,
 or equivalently to a finite-dimensional $\overline{\QQ}_p$-vector space equipped with an invertible endomorphism
(the \emph{linearized Frobenius action}). Note that the second equivalence depends on our prior choice of an embedding $j: K_n \hookrightarrow \overline{\QQ}_p$, but  the conjugacy class of the resulting endomorphism does not.

Let $X^\circ$ be the set of closed points of $X$. 
Given $\calE \in \FIsoc^\dagger(X) \otimes \overline{\QQ}_p$,
define the \emph{$L$-function} associated to $X$ as the power series
\[
L(\calE, T) := \prod_{x \in X^\circ} \prod_{\alpha \in S_x} (1 - \alpha T^{\deg(x/\FF_q)})^{-1} \in \overline{\QQ}_p \llbracket T \rrbracket,
\]
where $S_x$ is the multiset of eigenvalues of the linearized Frobenius action on $\calE_x$.

When $X$ is a curve, one can also define Euler factors at points of the smooth compactification of $X$;
these play a crucial role in the Langlands correspondence. See \cite[\S A.2]{abe-companion} for a detailed construction.
\end{defn}

\begin{theorem}[\'Etesse--Le Stum] \label{T:trace formula}
Suppose that $k = \FF_q$ is finite and that $X$ is of pure dimension $d$.
For $\calE \in \FIsoc^\dagger(X) \otimes \overline{\QQ}_p$, we have
\begin{equation} \label{eq:trace formula}
L(\calE, T) = \prod_{i=0}^{2d} \det(1 - q^{-d} F^{-1} T, H^i_{\rig}(X, \calE^\dual))^{(-1)^{i+1}}.
\end{equation}
\end{theorem}
\begin{proof}
See \cite[Th\'eor\`eme~6.3]{etesse-lestum}.
\end{proof}
\begin{remark} \label{R:trace formula}
In terms of cohomology with compact support, the Lefschetz trace formula for Frobenius 
reads more simply
\[
L(\calE, T) = \prod_{i=0}^{2d} \det(1 - FT, H^i_{c,\rig}(X, \calE))^{(-1)^{i+1}};
\]
moreover, it continues to hold without assuming that $X$ is smooth. See
\cite[Th\'eor\`eme~6.3]{etesse-lestum}, \cite[(2.1.2)]{kedlaya-weil2}.
\end{remark}

While it is not the case that overconvergent isocrystals can be described completely in terms of group representations
(except in the unit-root case), one can use the formalism of Tannakian categories to construct \emph{monodromy groups} that record some crucial information. We give a brief discussion here; see \cite{daddezio} for a more detailed exposition.

\begin{defn} \label{D:geometric monodromy}
Suppose that $X$ is connected and choose a closed point $x \in X^\circ$ (which we will typically neglect to mention when applying this construction). 
Let $\omega_x$ be the $\overline{\QQ}_p$-linear fiber functor on
$\FIsoc^\dagger(X) \otimes \overline{\QQ}_p$ taking $\calE$ to the underlying vector space of the linearized Frobenius action on $\calE_x$ (see Definition~\ref{D:linearized Frobenius action}). 
After restricting this to the Tannakian category generated by $\calE$ within $\FIsoc^\dagger(X) \otimes \overline{\QQ}_p$, we may extend it to the Tannakian category $[\calE]$ generated by $\calE$ within the category of \emph{overconvergent isocrystals on $X$ without Frobenius structure}
(tensored with $\overline{\QQ}_p$); this amounts to allowing objects which are stable under the connection but not the Frobenius action. (For a more thorough development of the theory of overconvergent isocrystals, see the references given in \S\ref{sec:basic construction}.)
 Taking the automorphism group of the resulting fiber functor yields a linear algebraic group $\overline{G}(\calE)$ over $\overline{\QQ}_p$, called the \emph{geometric monodromy group} of $\calE$.
\end{defn}

\begin{remark}
In the original development of monodromy groups of isocrystals given by Crew \cite{crew-mono},
the geometric hypotheses are somewhat stronger: $X$ is required to be a geometrically connected curve,
and the base point $x$ is required to be $k$-rational.
(See \cite[\S 3--4]{pal} for an alternate treatment in that context.)
The condition that $X$ be a curve was imposed due to limitations in the theory of overconvergent $F$-isocrystals at the time (in particular, Theorem~\ref{T:semistable} was unknown even for curves). 
The other restrictions were made so that Crew could work directly with $\FIsoc^\dagger(X)$,
and in particular to avoid having to require $k$ to be finite.
The extension of Crew's results to the setup we have described is straightforward, but for clarity we have chosen to spell out a few of the key steps.
\end{remark}

\begin{remark} \label{R:monodromy restriction}
Theorem~\ref{T:extend subobject} remains true for overconvergent isocrystals without Frobenius structure
(again see \cite[Proposition~5.3.1]{kedlaya-semi1}).
Consequently, geometric monodromy groups remain invariant under restriction from $X$ to an open dense subscheme;
this answers a question raised in \cite[Remark~2.9]{crew-mono}.
\end{remark}

For missing details in the following example, see \cite[Proposition~4.11]{crew-mono}. 
\begin{example} \label{exa:monodromy}
We calculate $\overline{G}(\calE)$ in the setting of Example~\ref{E:elliptic unit-root}.
By Remark~\ref{R:monodromy restriction}, this will not depend on whether we work over $X$ or $U$.

By definition, we have $\overline{G}(\calE) \subseteq \GL_2$.
Using the trace map in crystalline cohomology, one may show that $\overline{G}(\wedge^2 \calE)$ is trivial;
this implies that $\overline{G}(\calE) \subseteq \SL_2$.
On the other hand, the monodromy group of $\calE$ in the category of convergent isocrystals over $U$
is a Borel subgroup $B$ of $\SL_2$, corresponding to the slope filtration. (We omit the calculation that is required to show that the convergent monodromy group is not any smaller than $B$.) Since the slope filtration does not
extend to $\FIsoc^\dagger(X) \otimes \overline{\QQ}_p$, it follows that $\overline{G}(\calE) \neq B$ and hence $\overline{G}(\calE) = \SL_2$. In particular, $\Sym^n \calE$ is irreducible 
in $\FIsoc^\dagger(U) \otimes \overline{\QQ}_p$  for all positive integers $n$.
\end{example}

\begin{defn} \label{D:Weil group}
Fix a prime $\ell \neq p$. In the $\ell$-adic setting, the category corresponding to $\FIsoc^\dagger(X) \otimes \overline{\QQ}_p$ is the category of \emph{lisse Weil $\overline{\QQ}_\ell$-sheaves}. For $X$ connected, these may be described as the 2-colimit over finite extensions $L$ of $\QQ_p$ of the continuous representations of the \emph{Weil group} of $X$ on finite-dimensional $L$-vector spaces. The Weil group $W_X$ is in turn defined, in terms of a geometric point $\overline{x}$ of $X$, as the semidirect product of the geometric \'etale fundamental group $\pi_1(X_{\overline{k}}, \overline{x})$ by the action of Frobenius; we thus have an exact sequence
\[
1 \to \pi_1(X_{\overline{k}}, \overline{x}) \to W_X \to \ZZ \to 1.
\]
One may define the \emph{geometric monodromy group} and the \emph{Weil group} of an individual representation by taking the images of $\pi_1(X_{\overline{k}}, \overline{x})$ and $W_X$, respectively.

In a similar vein, we may define the \emph{Weil group} of $\calE \in \FIsoc^\dagger(X) \otimes \overline{\QQ}_p$ as 
the semidirect product $W(\calE)$ of $\overline{G}(\calE)$ by Frobenius; this is an algebraic group over $\overline{\QQ}_p$ equipped with a tautological linear representation which is faithful on $\overline{G}(\calE)$
and fitting into an exact sequence
\begin{equation} \label{eq:Weil group}
1 \to \overline{G}(\calE) \to W(\calE) \to \ZZ \to 1.
\end{equation}
The projection to $\ZZ$ is the \emph{degree map}.
\end{defn}

The definition of the Weil group makes it possible to transport various basic arguments from \'etale cohomology into the crystalline setting. As an example, we offer the following remark suggested by Marco D'Addezio.
(See \cite[Proposition~3.9]{pal} for an alternate treatment.)
\begin{remark} \label{R:finite base extension}
Suppose that $X$ is connected.
In \cite[\S 1.2, 1.3]{deligne-finite}, the following two facts about
lisse Weil $\overline{\QQ}_\ell$-sheaves are used without further comment; however, they are both contained in the conjunction of \cite[Lemma~II.2.4, Proposition~II.2.5]{yu}.
\begin{enumerate}
\item[(a)]
A representation $\rho$ of $W_{X_n}$ remains irreducible after all finite base extensions of $k$ if and only if its restriction to  $\pi_1(X_{\overline{k}}, \overline{x})$ is irreducible. In this case, we say $\rho$ is \emph{geometrically irreducible.}
\item[(b)]
Any irreducible representation splits, for some $n$, as a direct sum of $n$ irreducible representations of $W_{X_n}$ permuted cyclically by the $k$-Frobenius.
\end{enumerate}
Using similar arguments, we may obtain analogous assertions for isocrystals.
\begin{enumerate}
\item[(a)]
An object $\calE \in \FIsoc^\dagger(X) \otimes \overline{\QQ}_p$ remains irreducible after all finite base extensions of $k$ if and only if the tautological representation of $W(\calE)$ restricts to an irreducible representation of $\overline{G}(\calE)$ (i.e., if $\calE$ is irreducible even without its Frobenius structure). In this case, we say $\calE$ is \emph{geometrically irreducible}.
\item[(b)]
For $\calE \in \FIsoc^\dagger(X) \otimes \overline{\QQ}_p$ irreducible, there exists a positive integer $n$
such that in $\FIsoc^\dagger(X_{n}) \otimes \overline{\QQ}_p$, $\calE$ splits as a direct sum of $n$
geometrically irreducible summands  permuted cyclically by the action of the $k$-Frobenius.
\end{enumerate}
We spell this out in detail for (a). Suppose that $\calE$ remains irreducible after all finite base extensions of $k$.
Since $\calE$ is semisimple, $\overline{G}(\calE)$ is reductive (but see Corollary~\ref{C:semisimple} for a stronger statement). Since the tautological representation is faithful on
$\overline{G}(\calE)$, the implication (i) $\Rightarrow$ (iv) of \cite[Lemme~1.3.10]{deligne-weil2} implies that the center of $W(\calE)(\overline{\QQ}_p)$ contains an element $g$ of some positive degree $n$
(but again, see Corollary~\ref{C:center element positive degree} for a stronger statement).
By hypothesis, the tautological representation of $W(\calE)$ remains irreducible upon restriction to the inverse image of $n \ZZ$; since $g$ defines an automorphism of this restricted representation, by Schur's lemma it must act via a scalar multiplication. Hence the original representation of $\overline{G}(\calE)$ must also be irreducible.

For (b), we point out solely that the relevant arguments are
$1 \Rightarrow 2$ of \cite[Lemme~II.2.4]{yu} and $1 \Rightarrow 2$ of \cite[Proposition~II.2.5]{yu},
both of which are of purely group-theoretic nature; they concern the behavior of induction between Weil groups.
They thus carry over directly to arguments involving $W(\calE)$.
\end{remark}

One point at which the \'etale--crystalline analogy suffers some strain is the nature of abelian monodromy.
For a lisse Weil $\overline{\QQ}_\ell$-sheaf of rank 1 on $X$, the geometric monodromy group is always finite due to the mismatch between the $\ell$-adic and $p$-adic topologies \cite[Proposition~1.3.4]{deligne-weil2}.
This argument cannot be applied in the crystalline setting; however, it can be replaced with a 
 more intricate argument from geometric class field theory due to Katz--Lang \cite{katz-lang} to obtain the following result.
\begin{lemma}[Crew, Abe] \label{L:finite order determinant}
For any $\calE \in \FIsoc^\dagger(X) \otimes \overline{\QQ}_p$,
there exists an object $\calF \in \FIsoc^\dagger(k) \otimes \overline{\QQ}_p$ of rank $1$ such that
$\det(\calE \otimes \calF)$ is of finite order. In particular, if $X$ is connected, then $\overline{G}(\det(\calE))$ is finite. 
 \end{lemma}
\begin{proof}
This reduces at once to the case where $\rank(\calE) = 1$.
For this, see \cite[Corollary~1.5]{crew-mono} in the case where $X$ is a geometrically connected curve,
and \cite[Lemma~6.1]{abe-crelle} in the general case.
\end{proof}

The following corollary of Lemma~\ref{L:finite order determinant} is parallel to \cite[1.3]{deligne-finite}.
\begin{cor} \label{C:twist decomposition}
Suppose that $X$ is connected.
For $n$ a positive integer, let $\pi_n: X_n \to X$ be the canonical projection.
For any semisimple $\calE \in \FIsoc^\dagger(X) \otimes \overline{\QQ}_p$,
there exists a decomposition
\begin{equation} \label{eq:decompose to finite order}
\calE \cong \bigoplus_i \pi_{n_i*}(\calE_i \otimes \calL_i)
\end{equation}
in which for each $i$, $n_i$ is a positive integer,
$\calE_i$ is an object of $\FIsoc^\dagger(X_{n_i}) \otimes \overline{\QQ}_p$ which is geometrically irreducible (in the sense of Remark~\ref{R:finite base extension}) and has determinant of finite order,
and $\calL_i$ is an object of $\FIsoc^\dagger(k_{n_i}) \otimes \overline{\QQ}_p$ of rank $1$.
\end{cor}
\begin{proof}
We may assume at once that $\calE$ is irreducible.
By Remark~\ref{R:finite base extension}, there exists a positive integer $n$ such that $\calE$ splits in $\FIsoc^\dagger(X_n) \otimes \overline{\QQ}_p$ as a direct sum of absolutely irreducible subobjects
$\calF_1,\dots, \calF_n$ which are permuted cyclically by the action of Frobenius.
We then have a canonical isomorphism $\calE \cong \pi_{n*} \calF_1$, and applying Lemma~\ref{L:finite order determinant} to $\calF_1$ thus yields the desired result. 
\end{proof}

In general, the geometric monodromy group of an isocrystal cannot be naturally interpreted as a quotient of the geometric \'etale fundamental group; however, this is crucially true for unit-root isocrystals.

\begin{lemma} \label{L:unit-root monodromy}
Suppose that $X$ is connected and choose a geometric point $\overline{x}$
lying over $x$.
Let $\calE \in \FIsoc^\dagger(X) \otimes \overline{\QQ}_p$ be a unit-root object corresponding
as per Theorem~\ref{T:promotion unit-root} to a representation $\rho: \pi_1(X, \overline{x}) \to \GL_n(\overline{\QQ}_p)$.
Then there is a canonical isomorphism of $\overline{G}(\calE)$ with the Zariski closure of
$\rho(\pi_1(X_{\overline{k}}, \overline{x}))$.
\end{lemma}
\begin{proof}
This follows from Theorem~\ref{T:promotion unit-root} as in the proof of \cite[Proposition~3.7]{crew-mono}.
\end{proof}

This leads to an analogue of \cite[Proposition~4.3]{crew-mono}.
\begin{cor} \label{C:finite cover to trivial}
Suppose that $X$ is connected.
For $\calE \in \FIsoc^\dagger(X) \otimes \overline{\QQ}_p$, 
$\overline{G}(\calE)$ is finite if and only if
there exists a finite \'etale cover $f: X' \to X$ such that $\overline{G}(f^* \calE)$ is the trivial group.
(In the language of \cite{crew-mono}, this means that $f^*\calE$ is \emph{isotrivial}.)
\end{cor}
\begin{proof}
We may assume at once that $X$ is geometrically connected. Then the proof of \cite[Proposition~4.3]{crew-mono} carries over unchanged, except that \cite[Proposition~3.7]{crew-mono} must be replaced with Lemma~\ref{L:unit-root monodromy}.
\end{proof}

This in turn leads to an analogue of \cite[Proposition~4.6]{crew-mono}. 
\begin{lemma} \label{L:connected monodromy}
Suppose that $X$ is connected.
For $\calE \in \FIsoc^\dagger(X) \otimes \overline{\QQ}_p$,  the following statements hold.
(For $G$ a group scheme over a field, we write $G^\circ$ for the identity connected component of $G$.)
\begin{enumerate}
\item[(a)]
For $f: X' \to X$ a connected finite \'etale cover, 
the natural inclusion $\overline{G}(f^* \calE) \to \overline{G}(\calE)$ is an open immersion,
or in other words $\overline{G}(f^* \calE)^\circ = \overline{G}(\calE)^\circ$.
\item[(b)]
There exists a choice of $f$ for which $\overline{G}(f^* \calE)$ is connected, and therefore corresponds to $\overline{G}(\calE)^\circ$ via the natural inclusion.
\end{enumerate}
\end{lemma}
\begin{proof}
We may assume at once that $X$ is geometrically connected and choose a geometric basepoint $\overline{x}$ lying over $x$. To prove (a), it suffices to treat the case where $f$ is Galois (namely, for general $f$ we can find a Galois cover $f'$ factoring through $f$, and then the claim for $f'$ implies the claim for $f$).
We may also assume (by enlarging $k$ if needed) that both $X$ and $Y$ are geometrically connected.
Let $H$ be the automorphism group of $f$; then $H$ acts naturally on $[f^* \calE]$ and hence on $\overline{G}(f^* \calE)$, and the natural inclusion $\overline{G}(f^* \calE) \to \overline{G}(\calE)$ extends to a morphism $\overline{G}(f^* \calE) \rtimes H \to \overline{G}(\calE)$. From the fact that 
overconvergent isocrystals without Frobenius structure admit effective descent for finite \'etale coverings (see for example \cite{lazda}), 
it follows that $\overline{G}(f^* \calE) \rtimes H \to \overline{G}(\calE)$ is surjective,
and hence $\dim \overline{G}(f^* \calE) \geq \dim \overline{G}(\calE)$. Since $\overline{G}(f^* \calE) \to \overline{G}(\calE)$ is injective, we must have $\overline{G}(f^* \calE)^\circ = \overline{G}(\calE)^\circ$, proving (a).

To prove (b), we characterize the
quotient group $\pi_0(\overline{G}(\calE))$ as the automorphism group of $\omega_x$ on
the category of objects $\calF \in [\calE]$
for which $\overline{G}(\calF)$ is finite. 
This category can be generated by some finite set of irreducible objects $\calF_1,\dots,\calF_n$.
For each $i\in \{1,\dots,n\}$, $\calF_i$ occurs as a Jordan-H\"older constituent of an object $\calG$ of $\FIsoc^\dagger(X) \otimes \overline{\QQ}_p$; the set of isomorphism classes of constituents of $\calG$ in $[\calG]$ is finite and acted upon by $\varphi^*$, so there exists a positive integer $m$ such that
$\calF_i \oplus \varphi^* \calF_i \oplus \cdots \oplus \varphi^{(m-1)*} \calF_i$ admits a Frobenius structure.
We may thus apply Corollary~\ref{C:finite cover to trivial}
to $\calF_i \oplus \varphi^* \calF_i \oplus \cdots \oplus \varphi^{(m-1)*} \calF_i$ to obtain
a finite \'etale cover $f: X' \to X$ such that $\overline{G}(f^* \calF_i)$ is trivial.
By taking a fiber product, we can make a single choice of $f$ that works for all $i$; this cover has the desired effect.
\end{proof}

This finally leads to an analogue of Grothendieck's global monodromy theorem
\cite[Theorem~I.3.3(1)]{kiehl-weissauer}.
See also \cite[Theorem~3.4.4]{daddezio}.

\begin{theorem}[Crew] \label{T:global monodromy}
Suppose that $X$ is connected.
For $\calE \in \FIsoc^\dagger(X) \otimes \overline{\QQ}_p$, 
the radical of $\overline{G}(\calE)^\circ$ is unipotent.
\end{theorem}
\begin{proof}
It suffices to check the claim after replacing $X$ with a finite \'etale cover and/or replacing $k$ with a finite extension. By Lemma~\ref{L:connected monodromy}, we may thus assume that $X$ is geometrically irreducible, $x \in X(k)$, and $\overline{G}(\calE)$ is connected. We may then argue as in \cite[Theorem~4.9]{crew-mono},
using Lemma~\ref{L:finite order determinant} in place of \cite[Corollary~1.5]{crew-mono}.
\end{proof}
\begin{cor} \label{C:semisimple}
Suppose that $X$ is connected.
For $\calE \in \FIsoc^\dagger(X) \otimes \overline{\QQ}_p$ semisimple, $\overline{G}(\calE)$ is also semisimple.
\end{cor}
\begin{proof}
This follows from Theorem~\ref{T:global monodromy} as in  \cite[Corollary~4.10]{crew-mono}.
\end{proof}

\begin{cor} \label{C:center element positive degree}
Suppose that $X$ is connected and $\calE \in \FIsoc^\dagger(X) \otimes \overline{\QQ}_p$ is semisimple.
Let $Z$ be the center of $W(\calE)(\overline{\QQ}_p)$. Then the degree map $Z \to \ZZ$ 
has finite kernel and cokernel; more precisely, $Z$ contains a power of some element of $W(\calE)(\overline{\QQ}_p)$ of degree $1$.
\end{cor}
\begin{proof}
This follows from Corollary~\ref{C:semisimple} as in the proof of \cite[Theorem~I.3.3(2)]{kiehl-weissauer}.
\end{proof}

\section{Theory of weights}
\label{sec:weights}

Since rigid cohomology is a Weil cohomology theory, one may reasonably expect that the theory of weights in $\ell$-adic \'etale cohomology should carry over. This expectation turns out to be correct.

\begin{hypothesis} \label{H:weights}
Throughout \S\ref{sec:weights}, 
continue to retain Hypothesis~\ref{H:finite fields}.
In addition, fix an algebraic embedding $\iota: \overline{\QQ}_p \hookrightarrow \CC$.
\end{hypothesis}

\begin{defn}
Suppose that $\calE \in \FIsoc^\dagger(X) \otimes \overline{\QQ}_p$.
\begin{itemize}
\item
For $w \in \RR$, we say that $\calE$ is \emph{$\iota$-pure of weight $w$} if for each closed point $x \in X$ with residue field $k_n$, each eigenvalue $\alpha$ of the linearized Frobenius action on $\calE_x$
(see Definition~\ref{D:linearized Frobenius action}) satisfies $\left| \iota(\alpha) \right| = q^{nw/2}$.
\item
We say that $\calE$ is \emph{$\iota$-mixed of weights $\geq w$} (resp.\ $\leq w$)
if it is a successive extension of objects, each of which is $\iota$-pure of some weight $\geq w$ (resp.\ $\leq w$).
\end{itemize}
\end{defn}

We have the following partial analogue of Deligne's ``Weil II'' theorem \cite{deligne-weil2}. 
A more complete analogue can be stated in terms of constructible coefficients;
see \S\ref{sec:constructible}.
\begin{theorem}[Kedlaya] \label{T:weil2}
Suppose that $\calE \in \FIsoc^\dagger(X) \otimes \overline{\QQ}_p$ is $\iota$-mixed of weights $\geq w$.
Then for all $i \geq 0$, $H^i_{\rig}(X, \calE)$ is $\iota$-mixed of weights $\geq w+i$.
\end{theorem}
\begin{proof}
We may reduce to the case $\calE \in \FIsoc^\dagger(X)$,
for which see \cite[Theorem~5.3.2]{kedlaya-weil2}.
(The latter statement also includes a version for cohomology with compact supports, applicable without requiring $X$ to be smooth.)
\end{proof}

\begin{cor} \label{C:weight split}
Let 
\[
0 \to \calE_1 \to \calE \to \calE_2 \to 0
\]
be an exact sequence
in $\FIsoc^\dagger(X) \otimes \overline{\QQ}_p$ in which $\calE_i$ is $\iota$-pure of weight $w_i$,
$w_1 \neq w_2$, and $w_2 < w_1 + 1$. (In particular, these conditions hold if $w_2 < w_1$.) Then this sequence splits in $\FIsoc^\dagger(X) \otimes \overline{\QQ}_p$.
\end{cor}
\begin{proof}
We reduce formally to the case of an exact sequence in $\FIsoc^\dagger(X)$.
We have the following exact sequence of Hochschild-Serre type:
\[
0 \to H^0_{\rig}(X, \calE_2^\dual \otimes \calE_1)_F \to \Ext^1_{\FIsoc^\dagger(X)}(\calE_2, \calE_1) \to H^1_{\rig}(X, \calE_2^\dual \otimes \calE_1)^F.
\]
In this sequence, $H^0_{\rig}(X, \calE_2^\dual \otimes \calE_1)$ is finite-dimensional and $\iota$-pure of weight $w_1 - w_2 \neq 0$, so its Frobenius coinvariants are trivial.
Meanwhile, by Theorem~\ref{T:weil2}, $H^1_{\rig}(X, \calE_2^\dual \otimes \calE_1)$ is $\iota$-mixed of weights $\geq w_1 - w_2 + 1 > 0$, so its Frobenius invariants are also trivial.
\end{proof}

\begin{cor}[Abe--Caro] \label{C:weight filtration1}
Any $\calE\in \FIsoc^\dagger(X) \otimes \overline{\QQ}_p$ which is $\iota$-mixed admits a unique filtration
\[
0 = \calE_0 \subset \cdots \subset \calE_l = \calE
\]
such that each successive quotient $\calE_i/\calE_{i-1}$ is $\iota$-pure of some weight $w_i$, and $w_1 < \cdots < w_l$. We call this the \emph{weight filtration} of $\calE$.
\end{cor}
\begin{proof}
This is immediate from Corollary~\ref{C:weight split}. For an independent derivation
(and an extension to complexes), see \cite[Theorem~4.3.4]{abe-caro}.
\end{proof}

\begin{remark} \label{R:weight split equal case}
In Corollary~\ref{C:weight split}, half of the proof applies in the case $w_1 = w_2$:
the extension class in $H^1_{\rig}(X, \calE_2^\dual \otimes \calE_1)^F$ still vanishes.
We thus still get a splitting in the category of overconvergent isocrystals without Frobenius structures; consequently, any $\iota$-pure object in $\FIsoc^\dagger(X) \otimes \overline{\QQ}_p$ becomes semisimple in the category of overconvergent isocrystals without Frobenius structure. 
\end{remark}

\begin{remark}
While the proof of Theorem~\ref{T:weil2} draws many elements from Deligne's original arguments in \cite{deligne-weil2}, in overall form it more closely resembles the stationary phase method of Laumon \cite{laumon}, and
even more closely the exposition of Katz
\cite{katz-weil2} which makes some minor simplifications to Laumon's treatment.
In fact, translating the arguments from \cite{kedlaya-weil2} back to the $\ell$-adic side would yield an argument differing slightly even from \cite{katz-weil2}.

One pleasing feature of the $p$-adic approach is that the $\ell$-adic Fourier transform analogizes to a Fourier transform on some sort of $\mathscr{D}$-modules on the affine line, which is genuinely constructed by interchanging terms in a Weyl algebra. This point of view was originally developed by Huyghe \cite{huyghe}, and is maintained in \cite{kedlaya-weil2}.
\end{remark}

The following is analogous to a statement in the $\ell$-adic case which is a consequence of the Chebotarev density theorem; however, here one must instead make an argument using weights. 
\begin{theorem}[Tsuzuki] \label{T:chebotarev}
Suppose that $\calE_1, \calE_2 \in \FIsoc^\dagger(X) \otimes \overline{\QQ}_p$ are $\iota$-mixed and have the same set of Frobenius eigenvalues at each closed point $x \in X$. Then $\calE_1, \calE_2$ have the same semisimplification in $\FIsoc^\dagger(X) \otimes \overline{\QQ}_p$.
\end{theorem}
\begin{proof}
We reproduce the argument given in \cite[Proposition~A.4.1]{abe-companion}.
We may assume that $X$ is irreducible, and hence of some pure dimension $d$.
By Corollary~\ref{C:weight filtration1}, we may assume that $\calE_1, \calE_2$ are both $\iota$-pure, necessarily of the same weight $w$. For any irreducible $\calF \in \FIsoc^\dagger(X) \otimes \overline{\QQ}_p$,
we have $L(\calE_1 \otimes \calF^\dual, T) = L(\calE_2 \otimes \calF^\dual, T)$.
Combining Theorem~\ref{T:trace formula} with Theorem~\ref{T:weil2}, we see that for $j=1,2$, 
the pole order of $L(\calE_j \otimes \calF^\dual, T)$ at $T = q^{-d}$ equals
$\dim_{\overline{\QQ}_p} H^0_{\rig}(X, \calE_j \otimes \calF^\dual)^F$
(the factors in \eqref{eq:trace formula} with $i>0$ only contribute zeroes and poles in the region
$|T| \geq q^{-d+1/2}$).
The latter equals the multiplicity of $\calF$ as a constituent of $\calE_j$, so these agree for $j=1,2$ for all $\calF$; this proves the claim.
\end{proof}

By analogy with Deligne's equidistribution theorem,
one has an equidistribution theorem for Frobenius conjugacy classes
in rigid cohomology; this was described explicitly by Crew in the case where 
$\dim(X) = 1$ \cite[Theorem~10.11]{crew-finite}, but in light of the general theory of weights, one can adapt the proof of \cite[Th\'eor\`eme~3.5.3]{deligne-weil2}
to arbitrary $X$.

\begin{defn} \label{D:equidistribution}
Suppose that $X$ is connected and $\calE \in \FIsoc^\dagger(X) \otimes \overline{\QQ}_p$ is semisimple.
Set notation as in Definition~\ref{D:geometric monodromy} (with respect to some closed point $x \in X^\circ$).
Using the map $\iota$ to perform a base extension on the sequence \eqref{eq:Weil group}, as in \cite[\S 5]{crew-mono} we obtain an exact sequence
\[
1 \to G_\CC \to W_\CC \stackrel{\deg}{\to} \ZZ \to 1
\]
of affine $\CC$-groups.
There is a subgroup $W_\RR \subseteq W_\CC$ projecting onto $\ZZ$ such that
$G_\CC \cap W_\RR$ is a maximal compact subgroup of $G_\CC$ 
\cite[2.2.1]{deligne-weil2}. The conjugacy classes of $W_\RR$ are the intersections with $W_\RR$ of the conjugacy classes of $W_\CC$.

Choose any element $z$ of the center of $W_\RR$ of positive degree (see Remark~\ref{R:finite base extension} and Corollary~\ref{C:center element positive degree}).
Let $\mu_0$ be the measure on $W_\RR$ obtained as the product of Haar measure (normalized so $G_\RR$ has measure $1$) with the characteristic function of the set of elements of positive degree.

Let $W_\RR^\natural$ denote the space of conjugacy classes of $W_\RR$ equipped with the quotient topology. For any measure $\mu$ on $W_\RR$, let $\mu^\natural$ denote its image on $W_\RR^\natural$. For $n \in \ZZ$, let $W_n^\natural$ denote the set of classes in $W_\RR^\natural$ of degree $n$.

Suppose now that $\calE$ is $\iota$-mixed. 
As in \cite[2.2.6]{deligne-weil2}, for each closed point $x \in X$, 
we can find an element $g_x \in W_\RR$ conjugate in $W_\CC$ to a semisimplification of $\iota(\Frob_x)$.
Let $\mu$ be the measure on $W_\RR$ given by
\[
\mu = \sum_x  \deg(x) \sum_{n=1}^\infty q^{-n \deg(x)} \delta(g^n_x),
\]
where $\delta$ denotes a Dirac point measure.
\end{defn}

\begin{theorem} \label{T:equidistribution}
Suppose that $X$ is connected and $\calE \in \FIsoc^\dagger(X) \otimes \overline{\QQ}_p$ is semisimple and $\iota$-mixed. 
Then for any $i \in \ZZ$, in measure we have
\[
\lim_{n \to \infty} z^{-n} \mu^\natural|_{W^\natural_{i + n \deg(z)}} = \mu_0^\natural|_{W^\natural_i}.
\]
\end{theorem}
\begin{proof}
In light of Theorem~\ref{T:weil2} 
(or more precisely, its version for cohomology with compact supports,
to stand in for \cite[Corollaire~3.3.4]{deligne-weil2}), the proof of
\cite[Th\'eor\`eme~3.5.3]{deligne-weil2} applies unchanged.
\end{proof}

\begin{remark}
In the $\ell$-adic case, a more precise version of the equidistribution theorem has been formulated by Ulmer
\cite{ulmer}, although with few details of the proof. A more thorough argument, which also covers the $p$-adic case,
has been given by Hartl--P\'al \cite{hartl-pal}; this for example implies Zariski density of Frobenius conjugacy classes in the arithmetic monodromy group \cite[Theorem~4.13]{pal}.
\end{remark}

\begin{remark}
In the $\ell$-adic setting, one can refine the construction of local monodromy representations for lisse sheaves on a curve
(Remark~\ref{R:local mono rep}) to obtain \emph{local $\epsilon$-factors} which multiply together to give the global $\epsilon$-factor arising in the functional equation for the $L$-function; this was originally proved by Laumon \cite{laumon} building on work of Langlands and Deligne. Laumon's work admits a parallel version in the $p$-adic case,
as shown by Abe--Marmora \cite{abe-marmora}.
\end{remark}

\begin{remark}
One can also associate $L$-functions to convergent $F$-isocrystals, but the construction carries only $p$-adic analytic meaning; there is no theory of weights for such objects.
See
for example \cite{wan}.
\end{remark}

\section{A remark on constructible coefficients}
\label{sec:constructible}

To get any further in the study of rigid cohomology, one needs an analogue not just of lisse \'etale sheaves, but also constructible \'etale sheaves. Berthelot originally proposed a theory of \emph{arithmetic $\mathscr{D}$-modules} for this purpose
\cite{berthelot-d-modules}, and conjectured that \emph{holonomic} objects in this theory 
(equipped with Frobenius structure) are stable under the six operations formalism.
This result remains unknown, partly because the definition of holonomicity is itself a bit subtle; for instance, a direct arithmetic analogue of Bernstein's inequality fails, so one must use Frobenius descent to correct it.

In the interim, a modified definition of \emph{overholonomic} arithmetic $\mathscr{D}$-modules has been given by Caro \cite{caro-overholonomic}, as a way to formally salvage the six operations formalism. Of course, this provides little benefit unless one can prove that this category contains the overconvergent $F$-isocrystals as a full subcategory; fortunately, this is known thanks to a difficult theorem of Caro--Tsuzuki \cite{caro-tsuzuki} (whose proof makes essential use of Theorem~\ref{T:semistable}).
The theory of weights in Caro's formalism is developed in \cite{abe-caro}.

Recently, Le Stum has given a site-theoretic construction of overconvergent $F$-isocrystals
\cite{lestum} and proposed a theory of \emph{constructible isocrystals} \cite{lestum-constructible}. It is hoped that this again yields a six operations formalism, with somewhat less technical baggage required than in Caro's approach.

In any case, using arithmetic $\mathscr{D}$-modules, Abe \cite{abe-companion} has recently succeeded in porting L. Lafforgue's proof of the Langlands correspondence for $\GL_n$ over a function field 
\cite{lafforgue} into $p$-adic cohomology;
this immediately resolves Deligne's conjecture on crystalline companions \cite[Conjecture~1.2.10]{deligne-weil2} in dimension 1, and
ultimately leads to corresponding results in higher dimension. 
(Note that this requires working not just on schemes, but on certain algebraic stacks.)
See \cite{kedlaya-companions, kedlaya-companions2} for further discussion.

A related point is that the category of arithmetic $\mathscr{D}$-modules satisfies descent with respect to proper hypercoverings (as then does the category of overconvergent $F$-isocrystals). See \cite[\S 3]{abe-nearby}.

It is expected that one can similarly port V. Lafforgue's construction of (one direction of) the Langlands correspondence for any reductive group over a function field \cite{lafforgue-langlands}
into $p$-adic cohomology. This requires an adaptation of \emph{Drinfeld's lemma}, on products of fundamental groups in characteristic $p$,
 for both overconvergent $F$-isocrystals and arithmetic $\mathscr{D}$-modules. For discussion of the former, see \cite{kedlaya-drinfeld-lemma}.

\section{Further reading}
\label{sec:further reading}

We conclude with some suggestions for additional reading, in addition to the references already cited.
\begin{itemize}
\item
Berthelot's first sketch of the theory of rigid cohomology is the article 
\cite{berthelot-mem}; while quite dated, it remains a wonderfully readable introduction to the circle of ideas underpinning the subject.

\item
In \cite{kedlaya-aws}, there is a discussion of $p$-adic cohomology oriented towards machine computations, especially of zeta functions.

\item
In \cite{kedlaya-ams}, some discussion is given of how recent (circa 2009) results in rigid cohomology tie back to older results in crystalline cohomology.
\end{itemize}

\appendix

\section{Separation of slopes}

In this appendix, we record an alternate approach to Theorem~\ref{T:transversality} in the case $\calE \in \FIsoc(X)$ based on reduction to the local model statement, which is an unpublished result from the author's PhD thesis \cite[Theorem~5.2.1]{kedlaya-thesis}.
\begin{lemma} \label{L:transversality1}
Let
\[
0 \to \calE_1 \to \calE \to \calE_2 \to 0
\]
be a short exact sequence in $\FIsoc(k((t)))$ 
with $\calE_i$ isoclinic of slope $s_i$ and $s_2 - s_1 > 1$.
Then this sequence splits uniquely.
\end{lemma}
\begin{proof}
Using internal Homs, we may reduce to treating the case where $\calE_2$ is trivial, and in particular $s_2 = 0$ and $s_1 < -1$.
The extension group $\Ext^1_{\FIsoc(k((t)))}(\calE_2, \calE_1)$ may then be computed as the first total cohomology group of the double complex
\[
\xymatrix{
\calE_1 \ar^-{d/dt}[r] \ar^{\sigma-1}[d] & \calE_1 \ar^{pt^{p-1} \sigma-1}[d] \\
\calE_1 \ar^-{d/dt}[r] & \calE_1 
}
\] 
where the top left entry is placed in degree 0. A 1-cocycle is a pair $(\bv_1, \bv_2) \in \calE_1 \times \calE_1$ with $\frac{d}{dt}(\bv_1) = (pt^{p-1} \sigma-1)(\bv_2)$, and a 1-coboundary is a pair for which there exists an element $\bv \in \calE_1$ with $(\sigma-1)(\bv) = \bv_1$, $\frac{d}{dt}(\bv) = \bv_2$.

For $c>0$, let $\Gamma^{\perf(c)}$ be the subring of $\Gamma^{\perf}$ consisting
of those $x$ for which for each $n \geq 0$, there exists $y_n \in \Gamma$ such that
$\sigma^{-n}(y_n) - x$ is divisible by $p^{\lfloor cn \rfloor}$. Note that for $c>1$, the operator $\frac{d}{dt}$ on $\Gamma$ extends to a well-defined map $\Gamma^{\perf(c)}[p^{-1}] \to \Gamma^{\perf(c-1)}[p^{-1}]$. 

Since $\calE_1$ is isoclinic of slope $s_1< -1$, we may define
\[
\bv = \sigma(1 + \sigma^{-1} + \sigma^{-2} + \cdots)(\bv_1) \in \calE_1 \otimes_{\Gamma[p^{-1}]} \Gamma^{\perf(-s_1)}[p^{-1}]
\]
via a convergent infinite series. By the previous paragraph, we may then form
$\frac{d}{dt}(\bv) \in \calE_1 \otimes_{\Gamma[p^{-1}]} \Gamma^{\perf(-s_1-1)}[p^{-1}]$,
which satisfies
\[
(pt^{p-1} \sigma - 1)\left( \frac{d}{dt}(\bv) - \bv_2 \right) = 0.
\]
Since $s_1 + 1 < 0$, $pt^{p-1} \sigma - 1$ is bijective on $\calE_1 \otimes_{\Gamma[p^{-1}]} \Gamma^{\perf}[p^{-1}]$, so this forces 
\begin{equation} \label{eq:transversality1}
\frac{d}{dt}(\bv) = \bv_2.
\end{equation}
It will now suffice to check that this equality forces $\bv \in \calE_1$.

To see this, write $\Gamma^{\perf}[p^{-1}]$ as a completed direct sum of
$t^\alpha \Gamma[p^{-1}]$ with $\alpha$ varying over $\ZZ[p^{-1}] \cap [0,1)$,
then split $\calE_1 \otimes_{\Gamma[p^{-1}]} \Gamma^{\perf}[p^{-1}]$ accordingly.
For each component $t^{\alpha} \bv_\alpha$ of $\bv$ with $\alpha \neq 0$,
\eqref{eq:transversality1} then implies $\frac{d}{dt}(t^{\alpha} \bv_\alpha) = 0$.

Now let $\Gamma^{\unr}$ be the completion of the maximal unramified extension of $\Gamma$;
the derivation $\frac{d}{dt}$ extends uniquely by continuity to $\Gamma^{\unr}$.
By a suitably precise form of Theorem~\ref{T:unit root1} (e.g., see \cite[Corollary~5.1.4]{tsuzuki-finite}), there exists a basis $\be_1,\dots,\be_m$ of $\calE_1 \otimes_{\Gamma[p^{-1}]} \Gamma^{\unr}[p^{-1}]$ such that $\frac{d}{dt}(\be_i) = 0$ for $i=1,\dots,n$. Writing $\bv_\alpha = \sum_{i=1}^m c_i \be_i$ with
$c_i \in \Gamma^{\unr}[p^{-1}]$, we have
\begin{equation} \label{eq:transversality2}
0 = \frac{d}{dt}(t^{\alpha} \bv_\alpha)
= \sum_{i=1}^m t^{\alpha} \left( \alpha t^{-1} c_i + \frac{dc_i}{dt} \right) \be_i.
\end{equation}
However, the $p$-adic valuation of $\alpha$ is negative and the $p$-adic valuation of $c_i$ is no greater than that of its derivative, so \eqref{eq:transversality2} can only hold if $c_i = 0$ for all $i=0$. This implies that $\bv \in \calE_1$, as needed.
\end{proof}

\begin{lemma} \label{L:transversality2}
Theorem~\ref{T:transversality} holds in the case $\calE \in \FIsoc(X)$.
\end{lemma}
\begin{proof}
We first show that the claim may be reduced from $X$ to an open dense affine subspace $U$. The splitting of $\calE$ is defined by a projector, so it can be extended from $U$ to $X$ using 
Theorem~\ref{T:fully faithful1}. This in turn implies Theorem~\ref{T:transversality}(a) using Theorem~\ref{T:polygon variation}: the sum of the slopes of $\calE_1$ is locally constant, the largest slope of $\calE_1$ can only decrease under specialization, and the smallest slope of $\calE_2$ can only increase under specialization.

Using Theorem~\ref{T:polygon variation} again, we may thus reduce to the case where $\calE$ has constant slope polygon (so we no longer need to verify Theorem~\ref{T:transversality}(a) separately). By Corollary~\ref{C:filtration}, $\calE$ now admits a slope filtration. We are thus reduced to showing that if $X$ is affine and 
\[
0 \to \calE_1 \to \calE \to \calE_2 \to 0
\]
is a short exact sequence in $\FIsoc(X)$ 
with $\calE_i$ isoclinic of slope $s_i$ and $s_2 - s_1 > 1$, then this sequence splits uniquely.
Using Remark~\ref{R:etale pushforward} and Remark~\ref{R:affine cover}, we reduce to the case $X = \AAA^n_k$ (this is not essential but makes the argument slightly more transparent). As in Definition~\ref{D:convergent realization},
we may realize $\calE, \calE_1, \calE_2$ as finite projective modules over the Tate algebra $R = K \langle T_1,\dots,T_n \rangle$ equipped with compatible actions of the standard Frobenius lift $\sigma: T_i \mapsto T_i^p$ and the connection $\nabla$.
Let $R'$ be the completion 
of $K \langle T_1,\dots,T_n \rangle[T_1^{1/p^\infty}, \dots,T_n^{1/p^\infty}]$ for the Gauss norm; then the sequence of $\sigma$-modules splits uniquely over $R'$,
and we must show that this splitting descends to $R$ and is compatible with the action of the derivations $\frac{d}{dT_1},\dots,\frac{d}{dT_n}$. For this, we may apply Lemma~\ref{L:transversality1} to treat each variable individually.
\end{proof}

\end{document}